\documentclass[12pt]{amsart}
\usepackage{amssymb,latexsym,tikz}
\usepackage[top=40mm, bottom=25mm, left=35mm, right=35mm]{geometry}
\usetikzlibrary{arrows,decorations.markings}
\tikzstyle arrowstyle=[scale=1]
\tikzstyle directed=[postaction={decorate,
decoration={markings,mark=at position .65 with {\arrow[arrowstyle]{stealth}}}}]

\usepackage{multicol}
\usepackage{mathtools}
\usepackage{tikz-cd}
\usetikzlibrary{decorations.markings}
\tikzset{mid vert/.style={/utils/exec=\tikzset{every node/.append style={outer sep=0.8ex}},
postaction=decorate,decoration={markings,
  mark=at position 0.5 with {\draw[-] (0,#1) -- (0,-#1);}}},
mid vert/.default=0.75ex}
\usepackage{amscd}
\usepackage{xypic}
\usepackage[utf8]{inputenc}

\begin{document}
\title{Combinatorial Hopf algebras from restriction species with preorder cuts}




\newcommand{\llin}{\raisebox{1pt}{\scalebox{1}[0.6]{$\mid$}}}


\theoremstyle{plain}
\newtheorem{theorem}{Theorem}[section]
\newtheorem{corollary}[theorem]{Corollary}
\newtheorem*{main}{Main Theorem}
\newtheorem{lemma}[theorem]{Lemma}
\newtheorem{proposition}[theorem]{Proposition}
\newtheorem{conjecture}[theorem]{Conjecture}
\newtheorem{theoremp}{Theorem}

\theoremstyle{definition}
\newtheorem{definition}[theorem]{Definition}
\newtheorem{fact}[theorem]{Fact}
\newtheorem{obs}[theorem]{Observation}
\newtheorem{definisjon}[theorem]{Definisjon}
\newtheorem{problem}[theorem]{Problem}
\newtheorem{condition}[theorem]{Condition}

\theoremstyle{remark}
\newtheorem{notation}[theorem]{Notation}
\newtheorem{remark}[theorem]{Remark}
\newtheorem{example}[theorem]{Example}
\newtheorem{claim}{Claim}
\newtheorem{observation}[theorem]{Observation}
\newtheorem{question}[theorem]{Question}


\newcommand{\psp}[1]{{{\bf P}^{#1}}}
\newcommand{\psr}[1]{{\bf P}(#1)}
\newcommand{\op}{{\mathcal O}}
\newcommand{\opw}{\op_{\psr{W}}}

\newcommand{\ini}[1]{\text{in}(#1)}
\newcommand{\gin}[1]{\text{gin}(#1)}
\newcommand{\kr}{{\Bbbk}}
\newcommand{\pd}{\partial}
\newcommand{\vardel}{\partial}
\renewcommand{\tt}{{\bf t}}


\newcommand{\coh}{{{\text{{\rm coh}}}}}


\newcommand{\modv}[1]{{#1}\text{-{mod}}}
\newcommand{\modstab}[1]{{#1}-\underline{\text{mod}}}

\newcommand{\sut}{{}^{\tau}}
\newcommand{\sumit}{{}^{-\tau}}
\newcommand{\til}{\thicksim}

\newcommand{\totp}{\text{Tot}^{\prod}}
\newcommand{\dsum}{\bigoplus}
\newcommand{\dprod}{\prod}
\newcommand{\lsum}{\oplus}
\newcommand{\lprod}{\Pi}

\newcommand{\La}{{\Lambda}}

\newcommand{\sirstj}{\circledast}

\newcommand{\she}{\EuScript{S}\text{h}}
\newcommand{\cm}{\EuScript{CM}}
\newcommand{\cmd}{\EuScript{CM}^\dagger}
\newcommand{\cmri}{\EuScript{CM}^\circ}
\newcommand{\cler}{\EuScript{CL}}
\newcommand{\clerd}{\EuScript{CL}^\dagger}
\newcommand{\clerri}{\EuScript{CL}^\circ}
\newcommand{\gor}{\EuScript{G}}
\newcommand{\cF}{\mathcal{F}}
\newcommand{\cG}{\mathcal{G}}
\newcommand{\cM}{\mathcal{M}}
\newcommand{\cE}{\mathcal{E}}
\newcommand{\cI}{\mathcal{I}}
\newcommand{\cP}{\mathcal{P}}
\newcommand{\cK}{\mathcal{K}}
\newcommand{\cS}{\mathcal{S}}
\newcommand{\cC}{\mathcal{C}}
\newcommand{\cO}{\mathcal{O}}
\newcommand{\cJ}{\mathcal{J}}
\newcommand{\cU}{\mathcal{U}}
\newcommand{\cQ}{\mathcal{Q}}
\newcommand{\cX}{\mathcal{X}}
\newcommand{\cY}{\mathcal{Y}}
\newcommand{\cZ}{\mathcal{Z}}
\newcommand{\cV}{\mathcal{V}}

\newcommand{\mm}{\mathfrak{m}}

\newcommand{\dlim} {\varinjlim}
\newcommand{\ilim} {\varprojlim}

\newcommand{\CM}{\text{CM}}
\newcommand{\Mon}{\text{Mon}}


\newcommand{\Kom}{\text{Kom}}


\newcommand{\EH}{{\mathbf H}}
\newcommand{\res}{\text{res}}
\newcommand{\Hom}{\text{Hom}}
\newcommand{\inhom}{{\underline{\text{Hom}}}}
\newcommand{\Ext}{\text{Ext}}
\newcommand{\Tor}{\text{Tor}}
\newcommand{\ghom}{\mathcal{H}om}
\newcommand{\gext}{\mathcal{E}xt}
\newcommand{\id}{\text{{id}}}
\newcommand{\im}{\text{im}\,}
\newcommand{\codim} {\text{codim}\,}
\newcommand{\resol}{\text{resol}\,}
\newcommand{\rank}{\text{rank}\,}
\newcommand{\lpd}{\text{lpd}\,}
\newcommand{\coker}{\text{coker}\,}
\newcommand{\supp}{\text{supp}\,}
\newcommand{\Ad}{A_\cdot}
\newcommand{\Bd}{B_\cdot}
\newcommand{\Fd}{F_\cdot}
\newcommand{\Gd}{G_\cdot}


\newcommand{\sus}{\subseteq}
\newcommand{\sups}{\supseteq}
\newcommand{\pil}{\rightarrow}
\newcommand{\vpil}{\leftarrow}
\newcommand{\rpil}{\leftarrow}
\newcommand{\lpil}{\longrightarrow}
\newcommand{\inpil}{\hookrightarrow}
\newcommand{\pils}{\twoheadrightarrow}
\newcommand{\projpil}{\dashrightarrow}
\newcommand{\dotpil}{\dashrightarrow}
\newcommand{\adj}[2]{\overset{#1}{\underset{#2}{\rightleftarrows}}}
\newcommand{\mto}[1]{\stackrel{#1}\longrightarrow}
\newcommand{\vmto}[1]{\stackrel{#1}\longleftarrow}
\newcommand{\mtoelm}[1]{\stackrel{#1}\mapsto}
\newcommand{\bihom}[2]{\overset{#1}{\underset{#2}{\rightleftarrows}}}
\newcommand{\eqv}{\Leftrightarrow}
\newcommand{\impl}{\Rightarrow}

\newcommand{\iso}{\cong}
\newcommand{\te}{\otimes}
\newcommand{\into}[1]{\hookrightarrow{#1}}
\newcommand{\ekv}{\Leftrightarrow}
\newcommand{\equi}{\simeq}
\newcommand{\isopil}{\overset{\cong}{\lpil}}
\newcommand{\equipil}{\overset{\equi}{\lpil}}
\newcommand{\ispil}{\isopil}
\newcommand{\vvi}{\langle}
\newcommand{\hvi}{\rangle}
\newcommand{\susneq}{\subsetneq}
\newcommand{\sgn}{\text{sign}}


\newcommand{\xd}{\check{x}}
\newcommand{\ortog}{\bot}
\newcommand{\tL}{\tilde{L}}
\newcommand{\tM}{\tilde{M}}
\newcommand{\tH}{\tilde{H}}
\newcommand{\tvH}{\widetilde{H}}
\newcommand{\tvh}{\widetilde{h}}
\newcommand{\tV}{\tilde{V}}
\newcommand{\tS}{\tilde{S}}
\newcommand{\tT}{\tilde{T}}
\newcommand{\tR}{\tilde{R}}
\newcommand{\tf}{\tilde{f}}
\newcommand{\ts}{\tilde{s}}
\newcommand{\tp}{\tilde{p}}
\newcommand{\tr}{\tilde{r}}
\newcommand{\tfst}{\tilde{f}_*}
\newcommand{\empt}{\emptyset}
\newcommand{\bfa}{{\mathbf a}}
\newcommand{\bfb}{{\mathbf b}}
\newcommand{\bfd}{{\mathbf d}}
\newcommand{\bfl}{{\mathbf \ell}}
\newcommand{\bfx}{{\mathbf x}}
\newcommand{\bfm}{{\mathbf m}}
\newcommand{\bfv}{{\mathbf v}}
\newcommand{\bft}{{\mathbf t}}
\newcommand{\bbfa}{{\mathbf a}^\prime}
\newcommand{\la}{\lambda}
\newcommand{\bfen}{{\mathbf 1}}
\newcommand{\bfe}{{\mathbf 1}}
\newcommand{\ep}{\epsilon}
\newcommand{\en}{r}
\newcommand{\tu}{s}
\newcommand{\Sym}{\text{Sym}}

\newcommand{\ome}{\omega_E}

\newcommand{\bevis}{{\bf Proof. }}
\newcommand{\demofin}{\qed \vskip 3.5mm}
\newcommand{\nyp}[1]{\noindent {\bf (#1)}}
\newcommand{\demo}{{\it Proof. }}
\newcommand{\demodone}{\demofin}
\newcommand{\parg}{{\vskip 2mm \addtocounter{theorem}{1}  
                   \noindent {\bf \thetheorem .} \hskip 1.5mm }}

\newcommand{\lcm}{{\text{lcm}}}


\newcommand{\dl}{\Delta}
\newcommand{\cdel}{{C\Delta}}
\newcommand{\cdelp}{{C\Delta^{\prime}}}
\newcommand{\dlst}{\Delta^*}
\newcommand{\Sdl}{{\mathcal S}_{\dl}}
\newcommand{\lk}{\text{lk}}
\newcommand{\lkd}{\lk_\Delta}
\newcommand{\lkp}[2]{\lk_{#1} {#2}}
\newcommand{\del}{\Delta}
\newcommand{\delr}{\Delta_{-R}}
\newcommand{\dd}{{\dim \del}}
\newcommand{\Del}{\Delta}

\renewcommand{\aa}{{\bf a}}
\newcommand{\bb}{{\bf b}}
\newcommand{\cc}{{\bf c}}
\newcommand{\xx}{{\bf x}}
\newcommand{\yy}{{\bf y}}
\newcommand{\zz}{{\bf z}}
\newcommand{\mv}{{\xx^{\aa_v}}}
\newcommand{\mF}{{\xx^{\aa_F}}}

\newcommand{\Symm}{\text{Sym}}
\newcommand{\pnm}{{\bf P}^{n-1}}
\newcommand{\opnm}{{\go_{\pnm}}}
\newcommand{\ompnm}{\omega_{\pnm}}

\newcommand{\pn}{{\bf P}^n}
\newcommand{\hele}{{\mathbb Z}}
\newcommand{\nat}{{\mathbb N}}
\newcommand{\rasj}{{\mathbb Q}}
\newcommand{\bfone}{{\mathbf 1}}

\newcommand{\dt}{\bullet}
\newcommand{\disk}{\scriptscriptstyle{\bullet}}

\newcommand{\cxF}{F_\dt}
\newcommand{\pol}{f}

\newcommand{\Rn}{{\mathbb R}^n}
\newcommand{\An}{{\mathbb A}^n}
\newcommand{\frg}{\mathfrak{g}}
\newcommand{\PW}{{\mathbb P}(W)}

\newcommand{\pos}{{\mathcal Pos}}
\newcommand{\g}{{\gamma}}

\newcommand{\Vaa}{V_0}
\newcommand{\Bp}{B^\prime}
\newcommand{\Bpp}{B^{\prime \prime}}
\newcommand{\bbp}{\mathbf{b}^\prime}
\newcommand{\bbpp}{\mathbf{b}^{\prime \prime}}
\newcommand{\bp}{{b}^\prime}
\newcommand{\bpp}{{b}^{\prime \prime}}

\newcommand{\oLa}{\overline{\Lambda}}
\newcommand{\ov}[1]{\overline{#1}}
\newcommand{\ovv}[1]{\overline{\overline{#1}}}
\newcommand{\tm}{\tilde{m}}
\newcommand{\po}{\bullet}

\newcommand{\surj}[1]{\overset{#1}{\twoheadrightarrow}}
\newcommand{\Supp}{\text{Supp}}

\def\CC{{\mathbb C}}
\def\GG{{\mathbb G}}
\def\ZZ{{\mathbb Z}}
\def\NN{{\mathbb N}}
\def\RR{{\mathbb R}}
\def\OO{{\mathbb O}}
\def\QQ{{\mathbb Q}}
\def\VV{{\mathbb V}}
\def\PP{{\mathbb P}}
\def\EE{{\mathbb E}}
\def\FF{{\mathbb F}}
\def\AA{{\mathbb A}}

\newcommand{\promap}{\lpil}

\newcommand{\kk}{{\Bbbk}}
\renewcommand{\SS}{{\mathcal S}}
\newcommand{\pr}{\preceq}
\newcommand{\su}{\succeq}
\newcommand{\bX}{{\mathbf X}}
\newcommand{\bY}{{\mathbf Y}}
\newcommand{\bA}{{\mathbf A}}
\newcommand{\bB}{{\mathbf B}}
\renewcommand{\op}{{\text {op}}}
\newcommand{\set}{{\rm \bf set}}
\newcommand{\poset}{{\rm \bf poset}}
\newcommand{\setx}{{\rm \bf set^{\times}}}
\newcommand{\setmm}{{\rm \bf set_{\NN}}}
\newcommand{\setz}{{\rm \bf set_{\ZZ}}}
\newcommand{\setii}{{\rm \bf set^{ci}}}

\newcommand{\vect}{{\rm \bf vect}}
\newcommand{\spsm}{{\mathsf {SpS}_{\NN}}}
\newcommand{\sps}{{\mathsf {SpS}}}
\newcommand{\spvv}{{\mathsf {SpV}}}
\newcommand{\lattice}{{\rm \bf lattice}}
\newcommand{\Pre}{{\rm Pre}}
\newcommand{\nill}{{\mathbf 0}}
\newcommand{\nil}{\nill}
\renewcommand{\CC}{\mathcal C}
\newcommand{\one}{{\mathbf 1}}
\newcommand{\scup}{\sqcup}
\newcommand{\cut}{\text{cut}}

\newcommand{\vlin}{\raisebox{1.4pt}{\scalebox{1}[0.45]{$\mid$}}}
\newcommand{\lli}{\raisebox{3pt}{\rule{3.3mm}{0.5pt}}}

\newcommand{\nrel}[1]{\,\mathrlap{{\hskip 1.8mm}{\vlin}}{\lli}_{\,#1}\,}
\renewcommand{\cc}{{\rm \bf {cc}}}
\newcommand{\nc}{{\rm \bf {nc}}}
\newcommand{\cn}{{\rm \bf {cn}}}
\newcommand{\nn}{{\rm \bf {nn}}}
\newcommand{\latt}{{\rm \bf {lattice}}}
\newcommand{\co}{{\rm {co}}}

\newcommand{\sS}{\mathsf{S}}
\newcommand{\sA}{\mathsf{A}}
\newcommand{\sSA}{\mathsf{S}_{\backslash{\sA}}}
\newcommand{\sBA}{\mathsf{B}_{\backslash{\sA}}}
\newcommand{\sPA}{\mathsf{P}_{\backslash{\sA}}}

\newcommand{\sPa}{\mathsf{Park}}
\newcommand{\sPaA}{\mathsf{Pa}_{\backslash{\sA}}}
\newcommand{\sP}{\mathsf{P}}
\newcommand{\sQ}{\mathsf{Q}}
\newcommand{\sB}{\mathsf{B}}
\newcommand{\fB}{\mathbf{B}}
\newcommand{\sR}{\mathsf{R}}
\newcommand{\sH}{\mathsf{H}}
\newcommand{\sT}{\mathsf{T}}
\newcommand{\sG}{\mathsf{G}}
\newcommand{\sCC}{\mathsf{CC}}
\newcommand{\sNC}{\mathsf{NC}}
\newcommand{\sCN}{\mathsf{CN}}
\newcommand{\sNN}{\mathsf{NN}}
\newcommand{\bob}{\circ}
\newcommand{\clK}{\overline{\mathcal{K}}}
\newcommand{\mato}[1]{\overset{#1}{\mapsto}}
\newcommand{\xlp}[1]{X_{\backslash #1}}
\newcommand{\xgp}[1]{X_{/ #1}}
\newcommand{\tot}{{\text {tot}}}
\newcommand{\unit}{{\mathbf 1}}

\newcommand{\oX}{\overline{X}}
\long\def\comment#1{}
\newcommand{\ignore}[1]{}

\newcommand{\ABCD}[4]{
\begin{tikzpicture}[scale = 0.7]
  \draw (-0.7,0)--(0.7,0);
  \draw (0,-0.7)--(0,0.7);
  \node at (-0.4,-0.4) {$#3$};
  \node at (-0.4,0.4) {$#1$};
  \node at (0.4,-0.4) {$#4$};
  \node at (0.4,0.4) {$#2$};
\end{tikzpicture}}

\newcommand{\STv}[2]{
\begin{tikzpicture}[scale = 0.7]
  \draw (0,-0.7)--(0,0.7);
  \node at (-0.4,0) {$#1$};
  \node at (0.4,0) {$#2$};
\end{tikzpicture}}

\newcommand{\STh}[2]{
\begin{tikzpicture}[scale = 0.7]
  \draw (-0.6,0)--(0.6,0);
  \node at (0,-0.4) {$#2$};
  \node at (0, 0.4) {$#1$};
\end{tikzpicture}}


\newcommand{\ABCDb}[4]{
\begin{tikzpicture}[scale = 0.7]
  \draw (-1,0)--(1,0);
  \draw (0,-1)--(0,1);
  \node at (-0.67,-0.55) {$#3$};
  \node at (-0.67,0.55) {$#1$};
  \node at (0.67,-0.55) {$#4$};
  \node at (0.67,0.55) {$#2$};
\end{tikzpicture}}

\newcommand{\STvb}[2]{
\begin{tikzpicture}[scale = 0.7]
  \draw (0,-0.7)--(0,0.7);
  \node at (-0.8,0) {$#1$};
  \node at (0.9,0) {$#2$};
\end{tikzpicture}}

\newcommand{\SThb}[2]{
\begin{tikzpicture}[scale = 0.7]
  \draw (-0.7,0)--(0.7,0);
  \node at (0,-0.4) {$#2$};
  \node at (0, 0.4) {$#1$};
\end{tikzpicture}}







\author{Gunnar Fl{\o}ystad}
\address{Matematisk institutt\\
        Postboks\\
        5020 Bergen\\
      Norway}
\email{gunnar@mi.uib.no}

 \begin{abstract}
  We get new Hopf algebras (HA):
  1. A wealth of quotient HA's of the Malvenuto-Reutenauer
  HA (the Loday-Ronco HA being a special case). They consist of the permutations
  avoiding an {\it arbitrary} set of permutations without global descents,
  2. A HA of pairs of parking filtrations,
  and 3. Four HA of pairs of preorders.

  New concepts in this setting are: 1. a category $\setmm$ whose objects
  are sets,
  but morphisms are represented by matrices of natural numbers, and 2.
  restriction species $\sS$ on sets coming with  
  pairs of natural transformations
  $\pi_1, \pi_2 : \sS \rightarrow \Pre$ to the species of preorders.
  These induce two coproducts $\Delta_1$ and $\Delta_2$. Dualizing
  $\Delta_1$ gives product $\mu_1$ and coproduct $\Delta_2$, 
  giving bimonoid species. 
\end{abstract}











\maketitle
\section{Introduction}
\label{sec:intro}
A basic tenet of {\it combinatorial} Hopf algebras is that they
come with a distinguished {\it basis} (or sometimes several).
The basis elements are typically isomorphism classes of combinatorial
objects. To make a more refined setting one may work with Hopf species
where one has labeled objects, and then derive a Hopf algebra by a suitable
functor (the Fock functor \cite[Sec.15]{AgMa}). A species
is a functor from $\setx$ (sets with bijections) to usually either the
category $\vect$ of vector spaces, or $\set$ the category of sets.

For combinatorial Hopf species using $\vect$ is common, but again
since they come with a distinguished basis, on the face of it
$\set$ appears more direct. Versions using $\set$ are found in
\cite[Def.2.5]{AgAr} and
in \cite[Sec.4]{AM2}. However drawbacks are:
\begin{itemize}
\item[i.] These are somewhat ad hoc defined, they are not purely categorical.
  In fact comonoids in species over $\set$
  do not exist in the categorical sense.
\item[ii.] The definition does not work for most Hopf algebras, as
a product of two basis elements is typically a sum of
basis elements.
\end{itemize}

\subsection{New concepts and points of view}
\label{subsec:intro-new}
Based on many of the most ubiquitous combinatorial Hopf algebras, we
introduce several new points of view:

\medskip
\noindent {\bf 1.}{\it  Replace $\vect$ with a category $\setmm$.} The objects
 in this category are sets,
  but have many more morphisms than $\set$, bringing it closer to
  $\vect$. In particular the empty set is a null object, and a morphism
  $X \mto{f} Y$ has a dual morphism $Y \mto{Df} X$. By this we may get:
  
\noindent {\bf 2.} {\it Two coproducts $\Delta^1, \Delta^2$.} 
  Given a bimonoid with coproduct $\Delta$ and product
  $\mu$, the product $\mu$ may now be dualized to a coproduct $\Delta^\prime$.
  We take two coproducts $\Delta^1, \Delta^2$ as our {\it starting point}, and
  ask: When do these give a bimonoid when you dualize $\Delta^2$?
  The rationale for this is that coproducts, splitting up, is usually
  simpler to work with than products,  all possible ways to assemble together.
So our focus turns to making coproducts:
  
  \noindent {\bf 3.} {\it Coproducts from restrictions.}
  The species $\sS$ we shall work with have restrictions and these will give the
  coproducts.   For every injection of
  sets $Y \hookrightarrow X$ there is a restriction map
  \[ \sS[X] \pil \sS[Y], \quad s \mapsto s_Y. \]
  For $X$ a disjoint union $A \sqcup B$, our coproduct will send:
  \begin{equation}
    \sS[X] \pil \sS[A] \times \sS[B], \quad s \mapsto \text{ restriction
      pair } (s_A,s_B) \text { or } \nil. \label{eq:intro-rest}
  \end{equation}
  This requires a decision whether to map $s$ to $\nil$ or not, leading
  to:
  
\noindent {\bf 4.} {\it Species over preorders.} There is a restriction species
  $\Pre$ over $ \set$ where $\Pre[X]$ is the set of preorders
  on $X$. We require our species $\sS$ to have a
  natural transformation $\pi : \sS \pil \Pre$.
  To each $s \in \sS[X]$ we get a preorder $\pi(s)$ on the set $X$.
  This preorder
  determines  whether the image of $s$ in \eqref{eq:intro-rest} is $(s_A,s_B)$
  or $\nil$ (see Subsection \ref{subsec:intro-bimonoid}.{4.}).

  \medskip
\subsection{New Hopf algebras introduced}
From the above we get several new classes of Hopf algebras
based on the following.

\noindent {\bf A.} {\it Avoidance of  permutations.}
Let $S = \cup_{n \geq 0} S_n$ be the union of permutations of all sizes.
  Let $A \sus S$ a set of permutations which have {\it no global descents},
  and $S_{/A} \sus S$ the set of $A$-avoiding permutations.
  We show the associated
  vector space $\kk S_{/A}$ is a quotient Hopf algebra of the
  Malvenuto-Reutenauer (MR) Hopf algebra, Theorem \ref{thm:perm-A}.
  Examples include:
  \begin{itemize}
  \item $A = \{ 213 \}$ gives the Loday-Ronco Hopf algebra,
  \item $A = \{213, 132\}$ gives the Hopf algebra of quasi-symmetric functions,
  \item $A = \{ 3142, 2413 \}$ gives the Hopf algebra
    ${\mathcal W}{\mathcal {PP}}$ of \cite{Fo-PP}.
  \end{itemize}

  \noindent {\bf B.} {\it Parking filtrations.} J-C. Novelli and J-Y.Thibon
  \cite{ThNo1} give a Hopf algebra of parking functions
  having the MR-algebra as a quotient Hopf algebra. We introduce a larger
  ``master'' Hopf algebra of pairs of parking filtrations,
  Section \ref{sec:parking},
  with the Hopf algebra of parking functions as a
  sub-Hopf algebra.
  
\noindent  {\bf C.} {\it Pairs of preorders.}
  We introduce three large Hopf algebras coming from bimonoid
  species $\sB$ where $\sB[X]$ consists of pairs of preorders $(P,Q)$ on
  $X$, Sections \ref{sec:pairs} and  \ref{sec:HApairs}.
  Such pairs come in four classes $\cc, \nc, \cn, \nn$, each giving
  a large ``master'' Hopf algebra.

  \medskip
  \subsection{Our approach to get bimonoid species}
  \label{subsec:intro-bimonoid}
We describe in more detail our approach to
{\bf 1, 2, 3, 4} from Subsection \ref{subsec:intro-new}.

\noindent {\bf 1a.} For sets $X$ and $Y$ in the category $\setmm $ a morphism
$f : X \pil Y$ is a map of sets $X \mto{f} \Hom(Y, \NN)$, associating
for each $x \in X$ a multisubset of $Y$. Such a map may be represented
by an $|X| \times |Y|$-matrix with entries in non-negative integers.
By transposing the matrix we get a dual map $Df : Y \pil X$.
The empty set is a null object in this category.

Combinatorial bimonoid species are then categorical bimonoid species
$\setx \pil \setmm$. To get antipodes and combinatorial Hopf species one
just needs to extend to $\setx \pil \setz$.

\noindent {\bf 1b.}  For species $\sS$ in vector spaces, a coproduct
\[ \Delta : \sS[X] \pil \sS[A] \times \sS[B] \] usually takes a basis element
to
a single pair of basis elements. But a product
\begin{equation} \label{eq:intro-prod}
 \mu :  \sS[A] \times \sS[B] \mto{} \sS[X]
\end{equation}
often takes a pair of basis
elements $(s_A,s_B)$ to a sum of basis elements. However this sum can
usually be identified with a set (or multiset) of basis elements, so
we can work with species $\setx \pil \setmm $.

\medskip
\noindent {\bf 2.}
In the category $\setmm$ (in contrast to vector spaces)
the product map $\mu$ in \eqref{eq:intro-prod}
may be dualized to a coproduct map (without dualizing the objects):
\[ \Delta^\prime : \sS[X] \mto{} \sS[A] \times \sS[B], \]

Thus if $\sS $ is a bimonoid, we get two coproducts $\Delta $ and
$\Delta^\prime$. One may then turn things around, start from
$\Delta $ and $\Delta^\prime  $ and inquire when do we get a bimonoid by
dualizing the latter. For restriction comonoids this is the notion of
{\it intertwined} coproducts, Definition \ref{def:species-inter}.

\medskip
\noindent {\bf 3.} Often our species $\sS$ has restrictions.
For an injection $Y \hookrightarrow X$ we have restriction maps
\[ \sS[X] \pil \sS[Y], \quad s \mapsto s_Y, \]
and the coproduct
\[ \sS[X] \pil \sS[A] \times \sS[B], \quad s \mapsto (s_A, s_B)
\text { or } \nil,\]
where $X = A \sqcup B$. This is the case for the Connes-Kreimer (CK) Hopf
algebra and the Malvenuto-Reutenauer Hopf algebra. For instance in
the CK-case a tree is mapped to a restriction pair $(t_A,t_B)$ if
$(A,B)$ is an admissible cut, and otherwise $t$ maps to $\nil $.
What must be decided is {\it when} to map to $\nil$.

\medskip
\noindent {\bf 4.} For this we introduce species over preorders.
Denote by $\setii$ the category of sets with coinjections $X \pil Y$ as
morphisms (i.e. $Y \hookrightarrow X$ is an injection). 
$\Pre : \setii \pil \set $ is the species with $\Pre[X]$ all preorders
on the set $X$. We work with restriction species  $\sS : \setii \pil \set $
with a natural transformation $\pi : \sS \pil \Pre $.
So for every $s \in  \sS[X]$ we have a preorder $\pi(s) \in \Pre[X]$. Then
$s$ maps to the restriction pair $(s_A,s_B)$ iff $(A,B)$ is a {\it cut}
for $\pi(s)$, otherwise $s$ maps to $\nil $. By a cut we mean that
$A$ is a down-set in $\pi(s)$ and $B$ its complement up-set.

A restriction species over preorders then gives a coproduct on species.
If there are {\it two} structures as species over preorders
$\pi_1, \pi_2 : \sS \pil \Pre$, we get two coproducts $\Delta^1$
and $\Delta^2$. It is usually simple and direct to verify if they
are intertwined, Subsection \ref{subsec:intertwined-species},
thus giving two bimonoid species by dualizing either
of these coproducts.

\subsection{Combinatorial Hopf algebras}
At the start of the introduction we stated that a combinatorial Hopf
algebra (CHA) comes with a distinguished basis. This is by many examples.
However we mention that in well-known
definitions of CHA's in the literature, bases are not explicitly required.
\cite{ABS} defines a CHA as a graded Hopf algebra with a distinguished
character. \cite{LR-CHA} requires the coalgebra to be isomorphic to
a cofree coalgebra, either cocommutative or coassociative.
A definition explicitly requiring a basis is in \cite[Def.3.3]{CEMM}.
We do not venture a definition of CHA's but note that our bimonoid
species being based on sets, gives a basis for the associated bi-algebras
or Hopf algebras.

\medskip
\subsection{Organization of the article.}

{\noindent \bf Part I: Restriction species in $\setmm$ 
  with preorder cuts}

Section \ref{sec:setcat} introduces the category $\setmm$ and
the notion of partial
pull-back diagrams for partial maps in this category. Section
\ref{sec:species}
    recalls the notions of i. species, ii. species with restrictions and
    iii. bimonoid species in this category.
Section \ref{sec:pre} recalls basic notions for preorders. We consider 
  the notion of global descents for pairs of total orders, and
  the notion of refinement of preorders.
  
  Section \ref{sec:bimonoid} gives the essential new idea of
  a restriction species over $\Pre$. It gives
  rise to a comonoid, and with two such structures we get two coproducts.
  The essential requirement for getting bimonoid species
  is that these coproducts are {\it intertwined}.
In Section \ref{sec:avoid}, for 
a subspecies $\sA$ of a restriction species $\sS$ we introduce the
  $\sA$-avoiding sub-species $\sSA$, and investigate when two intertwining
  coproducts for $\sS$ are still intertwined for $\sSA$.

  \medskip
  {\noindent \bf Part II: Constructions of Hopf algebras.}

  Section \ref{sec:respre} gives 
how Hopf algebras of polynomials, of tensors,
 of graphs and of preorders,
 come from restriction species over preorders.
 
  Section \ref{sec:perm} constructs a wealth of quotient Hopf algebras of the 
  Malvenuto-Reutenauer (MR) Hopf algebra. The MR-algebra may
  be viewed in two ways
  as coming from a restriction species over preorders. We give a main
  general consequence, Theorem \ref{thm:perm-A}: For
  sets of permutations without global descents, the avoiding
  subspecies of this gives quotient Hopf algebras of the MR Hopf algebra.
   This drops almost immediately out of our setting.
  An example case is the Loday-Ronco Hopf algebra.
  
Section \ref{sec:parking}
introduces parking filtrations and
  a new large master Hopf algebra consisting of pairs of
  parking filtrations. It has the Hopf algebra of parking functions \cite{ThNo1}
  as a subalgebra.

  \medskip
  {\noindent \bf Part III: Hopf algebras of pairs of preorders.}

Section \ref{sec:pairs}
considers species $\sS$ with $\sS[X]$ consisting of pairs
  of preorders $(P,Q)$ on $X$. We investigate when the resulting
  two natural projection
  structures as species over preorders give two intertwined
  coproducts. There are four basic types of such pairs: $\cc, \nc, \cn, \nn$,
  giving four species of preorders $\sCC, \sNC, \sCN$ and $\sNN$.
  These give four large master Hopf algebras.
  Section \ref{sec:HApairs}  describes
  the four basic types of pairs of preorders in more detail,
  and give examples of how we get other Hopf algebras by avoidance.

\medskip
\noindent {\it Acknowledgment.} I thank Dominique Manchon for
hosting and partially supporting my stay at Université de Clermont-Ferrand
during the fall semester
2021, where this work was initiated. I am grateful for feedback
on this article. I also thank Lorentz Meltzers
h{\o}yskolefond for partially supporting the stay.

\medskip
\noindent {\it Declarations of interests:} None.

\medskip \medskip
{\noindent \large \bf Part I:
  Restriction species in $\setmm$ with preorder cuts}

\section{Bimonoid  species in sets with
  multimaps}
\label{sec:setcat}

We give the category of sets with {\it multimaps}.
In this setting one has a null object, and one may
dualize maps. This brings us closer to vector spaces, while
our objects are still sets. We introduce the notion of {\it partial
pullback diagram} in this setting.

\subsection{The category of sets with multimaps}

\label{subsec:setmm}
Let $\NN = \{0,1, \ldots \}$ be the natural numbers.

\begin{definition}
  Let $X$ and $Y$ be sets. A {\it multimap} $f : X \promap Y$ 
  is a set map $X \pil \Hom(Y,\NN)$.
\end{definition}

This associates to each $x \in X$ a multiset in $Y$.
A map $\tau : Y \pil \NN$ sending  $y$ to $n_y$ is written
$\sum_{y \in Y} n_y y$. 
Two multimaps $X \overset{f}{\promap} Y$ and $Y \overset{g}{\promap} Z$
may be composed as follows. If 
\[ x \mato{f} \sum_{y \in Y} n^x_y y, \quad y \mato{g} \sum_{z \in Z}
  m^y_z z, \]
the composite sends
\[ x \xmapsto{g \circ f} \sum_{z \in Z} \sum_{y \in Y} n^x_y m^y_z z. \]
We can represent $f$ and $g$ by matrices and then $g \circ f$ is just
matrix multiplication.
To give a multimap $f : X \promap Y$
is equivalent to give a set map $X \times Y \pil \NN$,
or a multisubset of $X \times Y$. The symmetry $X \times Y \iso Y \times X$,
shows that we equivalently get a dual multimap $Df : Y \promap X$
(whose matrix is the transpose of $f$).
The multimap $f : X \promap Y$ is an ordinary map if for each $x$ the sum
$\sum_{y \in Y} n^x_y = 1$.

\begin{definition} The map $f$ is a {\it promap} if each 
$n^x_y$ is either $0$ or $1$.
This means that $f$ factors as $X \pil \Hom(Y, \{0,1 \}) \pil \Hom(Y, \NN)$.
The image of $x \in X$ may then be considered to be
a subset of $Y$.
The map $f$ is a {\it partial map} if for each $x$ the sum
$\sum_{y \in Y} n^x_y$ is $0$ or $1$
\end{definition}

The following is easily verified as multimaps are given
by matrices of non-negative integers.
\begin{lemma} \label{lem:setcat-bij}
  A multimap $f : X \pil Y$ is an isomorphism iff it is an ordinary
  map which is a bijection.
\end{lemma}

To give a promap $f$ is equivalent to
give a map $X \times Y \pil \{0,1\}$, or a subset of $X \times Y$ which is
  simply a relation between $X$ and
  $Y$. The dual map $Df : Y \promap X$ is then also a promap.
  The composition of two
  promaps will in general be a multimap and not a promap. However the
  composition of a promap and a partial map is a promap, and the
  composition of two partial maps is a partial map. 

Let $\setmm $ be the category whose objects are sets and whose morphisms
are multimaps. Note that it is equivalent to the category of free
commutative monoids via the association $X \rightsquigarrow \NN X$.
The empty set is both an initial object and a terminal
object in $\setmm$, so it is a null object. When the empty set is
considered to be in $\setmm$ we denote it as $\nill$.

This category $\setmm$ has all finite limits
and colimits. Both the product and coproduct are disjoint unions
of sets
\[ X \prod Y =  X \coprod Y = (X \times \{1\}) \cup (Y \times \{2\}). \]

It is furthermore a symmetric monoidal category with product
\[ X \times Y = \{ (x,y) \,| \, x \in X, y \in Y \} \]
where the set $\{* \}$ with one element is the unit.
The empty set, the null object
is absorbing for this product: $X \times \nill = \nill$.
If $X \mto{f} X^\prime$ and $Y \mto{g} Y^\prime$ are multimaps, we
get the multimap $X \times Y \mto{f \times g} X^\prime \times Y^\prime$ by
\begin{align*}
 X \times Y  \mto{f \times g}  \Hom(X^\prime, \NN) \times \Hom(Y^\prime, \NN)
  \pil & \Hom(X^\prime \times Y^\prime, \NN \times \NN)  \\
  \pil & \Hom(X^\prime \times Y^\prime, \NN),
\end{align*}
where the last map is induced by multiplication in $\NN$. 

Our point of view on Hopf species will be combinatorial rather than
algebraic. The category $\set$ is however not adequate, mainly because
it does not allow dualization of maps, and also does not have
a null object. The category $\setmm$ does however.
It is closer to the category
of vector spaces, while still being based on sets.

\medskip
\subsection{The functor to vector spaces}
Given a field $\kk$,
there is a functor to vector spaces over $\kk$: 
\[ F : \setmm \pil \vect \]
It sends a set $X$ to the free vector space $\kk X$. It sends a multimap
given by $X \overset{f}{\promap} Y$
sending $x \overset{f}{\mapsto} \sum_{y \in Y} n^x_y y$
to the linear map $F(f)$ sending
$x \mapsto \sum_{y \in Y} n^x_y y$. The duality $D$ on $\setmm$
and duality $(-)^*$ on vector spaces correspond so we have a commutative
diagram:
\[ \xymatrix{ \setmm \ar^{F}[r] \ar_{D}[d] & \vect \ar^{*}[d] \\ 
   \setmm \ar^{F}[r]  & \vect}.
\]

In \cite[Subsec.6.9]{AM2} there is also a contravariant functor
\[ G : (\setmm)^\op  \pil \vect \]
sending $X \pil \kk^X = \Hom(\kk X, \kk)$. (They define this functor on the
category $\set$ instead of $\setmm$.) This functor is the composition
$G = F \circ D = (-)^* \circ F$.

\subsection{Dualizing diagrams in $\setmm$}
Suppose we have given a diagram in $\setmm$ of multimaps:

\begin{equation} \label{eq:setcat-xyzw}
  \xymatrix{ X \ar^{\beta}@{->}[d] \ar^{\alpha}[r] &
    Y \ar@{->}[d]^{\gamma} \\
    Z \ar@{->}[r]^{\delta}  & W }\end{equation}
Dualizing respectively the horizontal and vertical maps we get diagrams:
\begin{equation} \label{eq:setcat-xyzwD}
\xymatrix{ X \ar[d]^{\beta} & Y \ar[d]^{\gamma} \ar[l]_{D\alpha} \\
  Z & W, \ar[l]_{D\delta} } \qquad
\xymatrix{ X \ar[r]^{\alpha} & Y \\
  Z \ar[u]_{D \beta} \ar[r]^{\delta} & W. \ar[u]_{D \gamma} }
\end{equation}

Recall that if $f : X \pil Y$ is a promap, each image $f(x)$
may be considered an element of the power set $P(Y)$. 

\begin{proposition} \label{pro:setcat-comm}
  Suppose all of $\alpha, \beta, \gamma, \delta$
  are promaps. The left diagram of \eqref{eq:setcat-xyzwD} is commutative
  iff in the diagram \eqref{eq:setcat-xyzw}
  for any $z \in Z$ and $y \in Y$, the
  two intersections (which are respectively in the power sets $P(W)$
  and $P(X)$): 
  \[ \delta(z) \cap \gamma(y), \quad D\beta(z) \cap D\alpha(y) \]
  have the same cardinalities.
\end{proposition}

\begin{proof}
  We have chosen $z \in Z$ and $y \in Y$. 
Consider the left diagram of \eqref{eq:setcat-xyzwD}. Assume it commutes.
  Let
  \[ D\delta \circ \gamma(y) = \beta \circ D\alpha(y) = \sum_{u \in Z} n_u u. \]
  Let $\gamma(y)$ considered as a subset of $W$ be $W^1 \cup W^2$ where
  $W^1$ consists of those $w$'s
  such that $z \in D\delta(w)$ and $W^2$ of those $w$'s such
  that $z \not \in D\delta(w)$.  Let furthermore $D\alpha(y)$
  considered as a subset of $X$ be $X^1 \cup X^2$
  where the images by $\beta$ of the elements in $X^1$ contain $z$ and
  the images of elements in $X^2$ do not contain $z$.
  Since all maps are promaps, the cardinalities $|W^1| = n_z = |X^1|$ and

  \[ W^1 = \delta(z) \cap \gamma(y), \quad X^1 = D\alpha(y) \cap D\beta(z).
  \]

  Conversely assume these sets always have the same cardinality, and fix $y$.
  For each $z$ let $n_z = |W^1| = |X^1|$.
  Then we see that
  \[ D\delta \circ \gamma(y) = \sum n_z z = \beta \circ D\alpha(y). \]
\end{proof}

  Suppose the maps $\alpha, \beta, \gamma, \delta$ are partial maps.
  Let $Y^\prime \sus Y$ be the subset where $\gamma $ is a set map.
  So $Y^\prime$ consists of those $y \in Y$ such that
  $\gamma(y) \neq \nil$. Similarly $Z^\prime \sus Z$ is the subset where
  $\delta$ is a set map.
  Furthermore let $X^{\prime \prime} \sus X$ be the subset where
  both $\alpha $ and $\beta$ are set maps, so $X^{\prime \prime}$ consists
  of those $x \in X$ such that both $\alpha(x), \beta(x) \neq \nil$.

\begin{definition} \label{def:setcat-pb}
  The diagram \eqref{eq:setcat-xyzw} is
  a {\it partial pullback diagram} if:
  \begin{itemize}
  \item[1.]  $X^{\prime \prime} \mto{\alpha} Y$ factors through $Y^\prime $ and
    $X^{\prime \prime} \mto{\beta} Z$ factors through $Z^\prime$
    \item[2.] The following diagram
  \begin{equation} \label{eq:setcat-pbp}
  \xymatrix{ X^{\prime \prime} \ar^{\beta}@{->}[d] \ar^{\alpha}[r] &
    Y^\prime \ar@{->}[d]^{\gamma} \\
    Z^\prime \ar@{->}[r]^{\delta}  & W }\end{equation}
is a pullback diagram in $\set$. (In particular the above \eqref{eq:setcat-pbp}
is a commutative diagram.)
\end{itemize}
\end{definition}

Note that the original  partial pullback diagram \eqref{eq:setcat-xyzw}
may not be commutative. It can happen that:
\begin{itemize}
  \item $\alpha(x) = \nil $ while
  $\delta \circ \beta(x) \neq \nil$, or
\item $\beta(x) = \nil$ while $\gamma \circ \alpha(x) \neq \nil$.
\end{itemize}

\begin{corollary} \label{cor:setcat-pullback}
Suppose $\alpha, \beta, \gamma, \delta$ are partial maps.
The left diagram of \eqref{eq:setcat-xyzwD} is commutative
iff
the   diagram \eqref{eq:setcat-xyzw} is
a partial pullback diagram.
\end{corollary}

\begin{proof}
 {\bf Part a.} Suppose the left diagram of \eqref{eq:setcat-xyzwD} commutes.
 Pick $x \in X$. 

 \noindent 1. Let $y = \alpha(x)$ and $z = \beta(x)$ both be nonzero.
  Since then $x \in D\beta(z) \cap D\alpha(y)$,
  the intersection $\delta(z) \cap \gamma(y)$ has cardinality $\geq 1$.
  Thus $\delta(z) = \gamma(y)$ is nonzero,
  and $\delta \circ \beta(x) =  \gamma \circ \alpha (x)$.

\noindent  2. If $\delta(z) = \gamma(y)$ is nonzero, the cardinality of their
intersection is one. So there is a unique $x$ in $D\beta(z) \cap D\alpha(y)$,
and so a unique $x$ such that $y = \alpha(x)$ and $z = \beta(x)$.

\medskip
\noindent{\bf Part b.} Conversely assume the diagram \eqref{eq:setcat-xyzw} is a
partial pullback diagram. Let $z \in Z$ and $y \in Y$.
Then $\delta(z) \cap \gamma(y)$ has cardinality $0$ or $1$.
If it is one, then since \eqref{eq:setcat-pbp} is
a pullback diagram, there is a unique $x \in X^{\prime \prime}$ such that
$z = \beta(x)$ and $y = \alpha(x)$. Then this $x$ is unique in
$D \beta(z) \cap D\alpha(y)$ and so this intersection also has cardinality
one.

If $\delta(z) \cap \gamma(y)$ has cardinality $0$, there can be no
$x$ such that $z = \beta(x)$ and $y = \alpha(x)$ since \eqref{eq:setcat-pbp}
is a partial pullback diagram. Hence $D \beta(z) \cap D\alpha(y)$ also has
cardinality $0$.
\end{proof}

\section{Species}
\label{sec:species}
When considering species in a setting relating to combinatorial
Hopf algebras, one normally has a species in vector spaces.
However in the combinatorial setting these vector spaces
come with a distinguished basis (or even several). This suggests
it could be more natural to consider bimonoid or Hopf species in the category
of sets. However this is too restrictive, as the notion of comonoid is not 
even defined for species in the category of sets.
The setting of $\setmm$, the category
of sets with {\it multimaps} is the fully satisfactory setting: One
has a null object, enabling comonoids, and one may dualize maps.

\subsection{Species in the category $\setmm$}
Let $\setx$ be the category of sets with bijections as morphism.
A species in a category $\CC$ is a functor
$\sP : \setx \pil \CC$. Species becomes again a category with natural
transformations as morphisms. We shall consider species where
$\CC$ is either $\set, \setmm$ or $\vect$. When $\CC$ is a monoidal category
with products, we can make species into a monoidal category.
For $\setmm$ and $\vect$ the monoidal products are respectively the
cartesian product $\times $ and the tensor product $\te$.

For $\setmm$ the monoidal product of species $\sP$ and $\sQ$ is defined as:
\[ (\sP \cdot \sQ)[X] = \coprod_{X = S \scup T} \sP(S) \times \sQ(T)\]
and the unit as:
\[ \unit[X] = \begin{cases} \nil, & X \neq \emptyset \\
    \{* \}, &  X = \emptyset \end{cases}. \]
The categories of species over $\set, \setmm$ and $\vect$
are denoted as
\[ \sps, \,\, \spsm, \,\, \spvv .\]
As usual in a monoidal category we then have the notion of a {\it monoid}.
It consists of natural transformations
\[ \mu : \sP \cdot \sP \pil \sP, \quad \iota : \unit \pil \sP. \]
For $\spsm$ the multiplication map  $\mu$ may for $X = S \scup T$
be decomposed as multimaps
\[ \mu_{S,T} : \sP[S] \times \sP[T] \promap \sP[X]. \]
The unit $i$ also has a non-trivial multimap 
\[ \iota_\emptyset : \{* \} \promap \sP[\emptyset] \]
which may be identified as an element of $\Hom(\sP[\emptyset], \NN)$. 
These maps fulfill axioms \cite[8.2.1]{AgMa}, the most significant being that
$\mu$ is associative, corresponding to the commutativity of the following
diagram:
\[ \xymatrix{ \sP[R] \times \sP[S] \times \sP[T] \ar[r]^{\hskip 2mm \id \times \mu_{S,T}}
    \ar[d]_{\mu_{R,S} \times \id} &
    \sP[R] \times \sP[S \scup T]  \ar[d]^{\mu_{R,S \scup T}} \\
    \sP[R \scup S] \times \sP[T] \ar[r]_{\mu_{R \scup S,T}} &
    \sP[R \scup S \scup T]
 }
\]
The dual notion is a comonoid in $\spsm$, \cite[8.2.2]{AgMa} consisting of
\[ \Delta : \sP \pil \sP \cdot \sP, \quad \epsilon : \sP \pil \unit \]
which decompose into multimaps
\[ \Delta_{S,T} : \sP[X] \promap \sP[S] \times \sP[T], \quad
 \epsilon :  \sP[X] \promap \begin{cases} \nil, & X \neq \emptyset \\
    \{*\}, & X = \emptyset.
  \end{cases}
\]
Note in the last case that a multimap $\sP[X] \promap \nil$ is an
ordinary (the unique) set map $\sP[X] \pil \{*\}$, and a multimap
$\sP[\emptyset] \promap \{*\}$ identifies as an ordinary set map
$\sP[\emptyset] \pil \NN$.
The maps $\Delta$ and $\epsilon$ fulfill requirements dual to those
for $\mu$ and $i$. The dualization functor
$\setmm \mto{D} \setmm$, turns a monoid into a comonoid and vice versa.

\begin{remark} \label{rem:setcat-gradeo}
  In our settings we will always assume $P[\emptyset] =
\{*\}$ and both $\iota_\emptyset$ and $\epsilon_\emptyset$ are given by the map
$ * \mapsto (* \mapsto 1)$.
\end{remark}

\medskip
If one instead of $\setmm$ considers the category $\set$ with usual maps
between sets, there is not a good notion of comonoid due to there
being no maps $\sP[X] \pil \emptyset$.
One has however in \cite[Sec.4]{AM2} the more ad hoc notion
of set-theoretic comonoid.

The categories $\set$ and $\setmm$ have the same objects, sets.
We may then consider the categories $\sps$ and $\spsm$
to also have the same objects.
The faithful functor $\set \pil \setmm$ induces a faithful
functor $\sps \pil \spsm$, so $\Hom_{\sps}(\sP, sQ) \sus
\Hom_{\spsm}(\sP,\sQ)$, with many more
morphisms in the latter $\Hom$-set. By Lemma \ref{lem:setcat-bij} two
objects in $\spsm$ are isomorphic iff they are isomorphic in $\sps$.

\subsection{Bimonoids and Hopf monoids}

A bimonoid in $\spsm$ is a species $\sB$, \cite[Subsec.8.3]{AgMa}
with a monoid structure
$(\mu,i)$ and a comonoid structure $(\Delta, \epsilon)$ such
that:
\begin{itemize}
  \item[1.] $\Delta,\epsilon$ are morphism of monoids, or equivalently
  \item[2.] $\mu,i$ are morphisms of comonoids.
\end{itemize}
    Concretely, let
$X = A \scup B \scup C \scup D$ be a partition into four sets.
We have a diagram below where all maps are isomorphisms and each of the
four positions simply indicate different ways of writing
$X$ as a union of these four sets. For instance at the lower
left position we mean $(A \scup B) \scup (C \scup D)$. 
\[ \xymatrix{ X
    \ar[d] \ar[r] &
    \raisebox{-6mm}{\STv{\begin{matrix} {A} \\ {C} \end{matrix}}
      {\begin{matrix} {B} \\ {D} \end{matrix}}} \ar[d] \\
    \raisebox{-5mm}{\STh{A \,B}{C \, D}} \ar[r]
    & \raisebox{-5mm}{\ABCD{A}{B}{C}{D}}}
\]

  This induces a diagram in $\setmm$, where at any of the four positions
  of the diagram, the dividing lines mean that we take the cartesian
  product of sets. For instance at the upper right position we have
  $\sB[A \scup C] \times \sB[B \scup D]$:
  \begin{equation} \label{eq:setcat-deltamu}
    \xymatrix{ \sB[X] \ar[d]_{\Delta_{A B, C D}}
      & \raisebox{-4mm}{\STvb{ \sB \left [
            \begin{matrix} A \\ C \end{matrix} \right ]}
        {\sB \left [ \begin{matrix} B \\ D \end{matrix} \right ]}}
      \ar[d]_{\Delta_{A,C}}^{\Delta_{B,D}} \ar[l]_{\mu_{A C, B D}} \\
    \raisebox{-6mm}{\SThb{\sB[A \, B]}{\sB[C \,  D]}}  
    &
\raisebox{-6mm}{\ABCDb{\sB[A]} {\sB[B]} {\sB[C]} {\sB[D]}}
\ar[l]_{\mu_{A,B}}^{\mu_{C,D}}
}.
\end{equation}

In the lower right position we really should have respectively
\[ \sB[A] \times \sB[B] \times \sB[C] \times \sB[D],
  \quad \sB[A] \times \sB[C] \times \sB[B] \times \sB[D] \]
and then applying the twist map $\beta :
\sB[B] \times \sB[C] \pil \sB[C] \times \sB[B]$. We omit this minor
detail here.
The requirements 1. and 2.  for a bimonoid species is that:
\begin{itemize}
\item The diagrams \eqref{eq:setcat-deltamu} commute,
\item The three diagrams below in $\setmm$ commute.
\end{itemize}
\[ 
\xymatrix{ \sB[\emptyset] \times \sB[\emptyset] 
  \ar[r]^{\epsilon \times \epsilon} \ar[d]^{\mu} & \{*\} \times \{*\} \ar[d]\\
  \sB[\emptyset] \ar[r]^{\epsilon} & \{* \} }, \quad
\xymatrix{ \{*\} 
  \ar[r]^{i} \ar[d]^{\Delta} & \sB[\emptyset] \ar[d]^{\Delta} \\
  \{*\} \times \{*\} \ar[r]^{i \times i}
  & \sB[\emptyset] \times \sB[\emptyset]  }
\]
and
\[ \xymatrix{ & \sB[\emptyset] \ar[dr]^{\epsilon} &  \\
    \{*\} \ar[ur]^{i} \ar[rr]^{\id} & & \{* \} }
\]

A {\it Hopf monoid} is a bimonoid $\sH$ with a morphism of species
$s : \sH \pil \sH$, the {\it antipode},
which acts as an inverse of the identity map in the
convolution monoid $\Hom(\sH,\sH)$.
In our setting this will normally not exist since our morphisms
have coefficients in the natural numbers. However we may simply extend
to integer coefficients. Let $\setz$ be sets with multimaps
with integer coefficients. Then a bimonoid in $\setmm$ becomes
a bimonoid in $\setz$. By
\cite[Prop. 8.10]{AgMa} a necessary
and sufficient condition for the existence of the antipode is
that we have a map $s_{| \emptyset} : \sH[\emptyset] \pil \sH[\emptyset]$
making $\sH[\emptyset]$ into a Hopf monoid in $\setz$.
In our cases this will hold as we will consider {\it connected}
bimonoids, i.e. which have $\sH[\emptyset] = \{*\}$.
We may take  $s_{|\emptyset}$ the identity map.

\medskip
\noindent {\it Note.} We shall only consider connected bimonoid species $\sB$,
i.e.
we have $\sB[\emptyset] = \{*\}$, a one-element set.
For bimonoids in species in $\setmm$, they then automatically become
Hopf monoids in species when we consider them as species in $\setz$, i.e.
extend the coefficients in morphisms
to integers. Our concern in this article will not be antipodes, and
we therefore stick to the minimal effective setting needed,
which is the category of species in $\setmm$. 

\medskip
\noindent{\it Significant point of view.}
For a bimonoid $\sB$ with $(\mu,i)$ and $(\Delta, \epsilon)$, the monoid
structure $(\mu,i)$ may be dualized to a comonoid structure
$(\Delta^\prime, \epsilon^\prime)$. Se we get $\sB$ with two comonoid structures.
We may then ask: Given a species $\sP$ with to comonoid structures
$(\Delta^1, \epsilon_1)$ and $(\Delta_2, \epsilon_2)$, when do they give
a bimonoid (and a Hopf) species when dualizing $(\Delta^2, \epsilon_2)$ to
$(\mu_2, i_2)$?

\medskip
\noindent{\it Important point.} If we had considered bimonoids $\sB$ in
species of vector spaces,
the dual of a product map $\mu : \sB[S] \te \sB[T] \pil \sB[X]$
will be a coproduct map  $\Delta^\prime : \sB[X]^* \pil \sB[S]^* \te \sB[T]^*$.
This does not really allow one to compare this with a coproduct
map $\Delta :  \sB[X] \pil \sB[S] \te \sB[T]$, since these
live on different (dual) spaces. This is a possible reason why this point of
view has apparently not been pursued. However when we consider species
in the category $\setmm$, the dual $\Delta^\prime$ of a product map $\mu$ is
a morphism between the same objects (sets) as the coproduct map $\Delta$.
This allows us to compare them.

\medskip
Very often for species, the maps $\Delta_{S,T}^i$ for $i = 1,2$
are {\it partial maps}.
So let us consider a species $\sP : \setx \pil \setmm$ with two
comonoid structures $\Delta^1$ and $\Delta^2$ such that for each
disjoint union $X = S \scup T$ the maps:
\[ \Delta^1, \Delta^2 : \sP[X] \promap \sP[S] \times \sP[T] \]
are partial maps. Let $(\mu_2, i_2)$ be the monoid
structure dual to $(\Delta^2, \epsilon_2)$.
For each decomposition $X = A \scup B \scup C \scup D$
  there is a diagram:
  \begin{equation} \label{eq:setcat-dia}
    \xymatrix{ \sP[X] \ar[r]^{\Delta^2_{A C,
          B D} \hskip 3mm} \ar[d]_{\Delta^1_{A B, C D}}
      & \raisebox{-3.5mm}{\STvb{\sP\left [ \begin{matrix} A \\ C \end{matrix} \right ]}
        {\sP\left [\begin{matrix} B \\ D \end{matrix} \right ]}
          \ar[d]^{\Delta^1_{B,D}}_{ \Delta^1_{A,C}}} \\
        \raisebox{-5mm}{\SThb{\sP[A \, B]}{\sP[C \, D]}
          \ar[r]^{\Delta^2_{A,B}}_{\Delta^2_{C,D}}}
          & \raisebox{-6.5mm}{\ABCDb{\sP[A]}{\sP[B]}{\sP[C]}{\sP[D]}}.
        }
\end{equation}

      We shall now investigate conditions on $\Delta^1$ and $\Delta^2$
      for this diagram such that $(\Delta^1, \mu_2, \epsilon^1, \iota_2)$
      becomes a bimonoid (and so a Hopf monoid with the assumptions above
      on $\sP[\emptyset]$ that this is $\{* \}$).

      \begin{proposition} \label{pro:setcat-DeMu}
        Let $\sP$ be a species in $\setmm$ with comonoid structures
        $\Delta^1, \Delta^2$ being partial maps,
        with $\epsilon_1 = \epsilon_2$ the natural maps on $\sP[\emptyset]$.
        Then  $(\sP, \Delta^1, \epsilon_1, \mu_2, i_2)$
        is a bimonoid in species iff
        the diagram \eqref{eq:setcat-dia} is always a partial pull-back
        diagram.
        \end{proposition}

    \begin{proof}
      This follows from Corollary \ref{cor:setcat-pullback}.
    \end{proof}

    \subsection{Species with restrictions}
    In all examples and cases we consider for species,
    for a subset $Y \sus X$ it is meaningful
    to have restriction maps $\sS[X] \pil \sS[Y]$.
    For two sets $X,Y$, a {\it coinjection} $X \pil Y$ is an
    injective map $i : Y \pil X$. 

    Let $\setii $ bf the category of finite sets with coinjections $X \pil Y$
    as morphisms. Note that $\setx$ embeds as a full subcategory by
    sending a bijection $\sigma : X \overset{\iso}{\pil} Y$ to the
    coinjection $\sigma : X \pil Y$ corresponding to the injection
    $\sigma^{-1} : Y \pil X$. 

    \begin{definition} A {\it species with restrictions} is
      a functor $\sS : \setii \pil \set $.
      An injection 
      $i : Y \pil X$ gives  a map $i^* : \sS[X] \pil \sS[Y]$,
      and for $s \in \sS[X]$ 
      we usually write $s_{Y}$ for $i^*(s)$. 
\end{definition}

Species with restrictions, as well as several other interesting variants,
were first considered in \cite{Sch-CoHa}.
They are discussed in \cite[Section 8.7.8]{AgMa}, relating to linearized
comonoids. The category of species of sets with restrictions is given
to be equivalent to linearized cocommutative comonoids, \cite[Prop.8.29]{AgMa}.
With the central idea of the present article, see further Section
\ref{sec:respre}, we may linearize set species to non-cocommutative
species. In \cite{Pe} species with restrictions are termed combinatorial
presheaves.

A species with restrictions $\sS : \setii \pil \set$
induces of course  a species with
restrictions $\sS : \setii \pil \setmm $.

\begin{definition} \label{def:setcat-resco}
  Let $\sS$ be a species with restrictions in $\set$
  and consider it as a species with restrictions in $\setmm$. 
  A comonoid structure $(\Delta,\epsilon)$ on $\sS$
  is a {\it restriction comonoid} iff for each decomposition
  $X = U \scup V$ in the comonoid map 
  \[ \Delta : \sS[X] \pil \sS[U] \times \sS[V], \] an element
  $s \in \sS[X]$ has image either $\Delta(s) = \nil$ or
  $\Delta(s) = (s_U,s_V)$ where $s_U$ and $s_V$ are the restrictions of $S$.

  The counit $\epsilon(s) = \nil$ if $X \neq \emptyset$, and $\epsilon(*)$
  is $* \in \unit[\emptyset]$ when $X = \emptyset$ (see Remark
  \ref{rem:setcat-gradeo}).
\end{definition}

\fbox{\parbox{11.8cm}
{{\it Central theme.} Given a restriction species,
it may induce several
restriction comonoids by varying when $\Delta(s)$ is
$\nil$ or not.} }

\medskip
\noindent Section \ref{sec:bimonoid} gives our general way of doing this.
In Sections \ref{sec:respre} - \ref{sec:HApairs} we for each species give
two comonoids $\Delta^1$ and $\Delta^2$.

\section{The lattice of preorders}
\label{sec:pre}

Our main notion in this article is restriction species over preorders.
We recall basics on preorders.

\subsection{Preorders} \label{subsec:preorders}
A {\it preorder} $\leq$ on a set $X$, is a relation on $X$ which is: 

\begin{itemize}
\item[1.] {\it Reflexive:} $a \leq a$ for every $a \in X$
\item[2.] {\it Transitive:} If $a \leq b$ and $b \leq c$ then $a \leq c$ for
  $a,b,c \in X$.
\end{itemize}

Denote by $\Pre(X)$ the set of all preorders on $X$.
It is a partially ordered set with order relation $\preceq $
where
\[ \leq_1 \,\, \preceq \,\,
  \leq_2 \text{ iff } a \leq_1 b \Rightarrow a \leq_2 b \text{ for every }
a,b \in X.\]
In fact $(\Pre(X), \preceq)$ is a complete lattice. The smallest element is the
discrete order where no two distinct elements are comparable and
the largest element is the coarse order $\leq $ where $a \leq b$ for every
pair of elements $a,b \in X$.

\begin{definition}
An {\it down-set} for a preorder $\leq$ on $X$, is a subset
of $I$ of $X$ such that $y \in I$ and $x \leq y$ implies $x \in I$.
An {\it up-set} for $\leq$ is a subset $F$ of $X$ such
that $x \in F$ and $x \leq y$ implies $y \in F$. Whenever $I$ is a down-set
the complement set $F = X \backslash I$  is an up-set. Such a pair
$(I,F)$ is a {\it cut} for the preorder $\leq$. 
\end{definition}

Preorders on $X$ correspond to finite topologies on $X$.
Given a preorder $P$, the open subsets of the topology
are the up-sets of $P$. So the open subsets are the upper sets
of cuts of $P$. In particular the {\it discrete topology} on $X$, $D = D(X)$, 
corresponds to the minimal element in $\Pre(X)$,  the preorder
where the only comparability
relations are $a \leq a$ for $a \in X$. The {\it coarse topology}, $C = C(X)$, 
where the only open subsets are $\emptyset$ and $X$, corresponds
to the maximal element in $\Pre(X)$, where
we have $a \leq b$ for any two elements $a,b$ in $X$.

\medskip
For two preorders
$P$ and $Q$ the meet
$P \wedge Q$ is the preorder $\leq$ where $a \leq b$ if $a \leq_P b$
and $a \leq_Q b$.
The join $P \vee Q$ is the unique smallest preorder $R$
such that $P \preceq R$ and $Q \preceq R$. 
It is obtained by taking the transitive closure of the union of
the ordering relations for $P$ and $Q$.
There is furthermore an involution on $\Pre(X)$ by sending a preorder
$P$ to its opposite preorder $P^\op$ where $a \leq_{P^\op} b$ iff
$b \leq_P a$. 

\begin{lemma} \label{lem:pre-resY}
  Let $P$ and $Q$ be preorders on $X$ and
  $Y \sus X$. Then:
  \[ P_{|Y} \vee Q_{|Y} \preceq (P \vee Q)_{|Y}, \quad
    P_{|Y} \wedge Q_{|Y} = (P \wedge Q)_{|Y}. \]
  Furthermore if $Y$ is either a down-set or an up-set for $P \vee Q$,
  we have equality in the first relation.
\end{lemma}

\begin{proof}
  It is clear that both $P_{|Y}, Q_{|Y} \preceq (P \vee Q)_{|Y}$, giving
  the first relation. The second is also clear.
  For the last statement, suppose $a \leq b$ in $(P \vee Q)_{|Y}$.
  Then there is a sequence of elements in $X$:
  \[ a = x_0 \leq_P x_1 \leq_Q x_2 \leq_P \cdots \leq_Q x_r = b. \]
  If $Y$ is a down-set for $P \vee Q$, it is a down-set for both
  $P$ and $Q$.  Since $b = x_r \in Y$ we
  get $x_{r-1} \in Y$ and then successively all $x_i$ in $Y$.
  Hence we also have $a \leq b$ in $P_{|Y} \vee Q_{|Y}$.
  The argument when $Y$ is an up-set is similar.
\end{proof}

\begin{definition} \label{def:pre-c} Given a preorder $P$ on $X$.
Define a relation on $X$ by $a \circ b$ if $a \leq b$ and $b \leq a$.
This is an equivalence relation, and so it partitions $X$ into
equivalence classes. Each such class is a {\it bubble} for $P$.
The preorder $P$ is a poset if each bubble is a singleton.

If $a \leq b$ and $a$ and $b$ are not in the same bubble, we write $a < b$. 
If $a$ and $b$ are elements that are not comparable for the preorder
we write $a \nrel{} b$. 

We have a {\it total preorder} if any two elements of $X$ are comparable
for the preorder. If the total preorder is a poset it is a {\it total order}.
In this case for any two elements $a,b$ in $X$ precisely one
of the following holds: $a < b, a = b, a > b$.

The preorder is a {\it partition order} if $a \leq b$ implies
$a \circ b$. In other words only elements in the same bubble are comparable.
Such a preorder is the same as an equivalence relation.
Note that a preorder $P$ is a partition order iff $P = P^{\op}$. 
\end{definition}

\begin{lemma} Let $P$ be a preorder.

a. $P^\bob = P \wedge P^{\op}$ is the partition order whose bubbles are the
bubbles of $P$. In particular $P$ is a poset iff $P^\bob$ is the
minimal element $D$.

b. $P^\bullet = P \vee P^{\op}$ is the partition order where the underlying sets
of the bubbles are the underlying sets of the connected components of
$P$. In particular if $P$ has only one component, then $P^\bullet$ is
the maximal element $C$.
\end{lemma}

\begin{proof}
Clearly $P^\bob = (P^\bob)^{\op}$, and so $P^\bob$ is a partition order.
The same goes for $P^\bullet$. It is also clear that $a$ and $b$ are in the
same bubble of $P$ iff they are in the same bubble of $P^\bob$.
Also $a$ and $b$ are in the same component of $P$ iff they are in the
same bubble of $P^\bullet$. 
\end{proof}

\subsection{Permutations and global descents} The
following will be used in Section \ref{sec:perm} on
permutation Hopf algebras.
 Let $T_1$ and $T_2$ be total orders on $X$. Then $T_1$ induces
  an order preserving bijection $X \pil [n]$. With this identification,
  the pair $(T_1,T_2)$ identifies as a permutation $\sigma$ on $[n]$.

  If one instead use $T_2$ to give an order preserving bijection
  $X \pil [n]$, the pair $(T_2,T_1)$ identifies as the inverse
  permutation $\sigma^{-1}$.

  \begin{lemma} Let $T_1$ and $T_2$ be total orders, with $(T_1,T_2)$
    corresponding to the permutation $\sigma$. 

a. The preorder $T = T_1 \vee T_2^{\op}$ is a total preorder.
  It corresponds
  to the global descent decomposition of $\sigma$: If we write
  \[ \sigma = n_1n_2 \cdots n_{r-1}| n_{r} \cdots n_{s-1} | n_s \cdots \]
  where we have global descents at  $r-1, s-1, \cdots$ and so on, then
  each of the segments between the vertical markers $|$ become
  bubbles in $T$ and where we have global descents, the ordering relation
  in $T$ is strict.

b. For the components of $S = T_1 \wedge T_2$, their underlying
  sets are  precisely the underlying sets of
  the bubbles in $T$, or formally $S^\bullet = T^\bob$. 
\end{lemma}

\begin{proof}
  {\bf $(T_1,T_2)$ has no global descent.}
  First assume the permutation $\sigma$
  corresponding to $(T_1,T_2)$ has no global
  descents. We show that $T = T_1 \vee T_2^\op$ is the maximal preorder $C$.
  Let $x \leq_1 y$. Since we have no global descent at $x$ there are
  $z \leq_1 x <_1 w$ such that $z \leq_2 w$, and so $z \geq_2^\op w$.
  If $y \leq_1 w$ we get $x \leq_1 y \leq_1 w \leq_2^\op z \leq_1 x$, and
  get $x$ and $y$ in the same $T$-bubble. If $w \leq_1 y$, we get
  $x \leq_1 w \leq_2^\op z \leq_1 x$ and get $x,w$ in the same bubble.
  Then we may continue as above with the pair $w \leq_1 y$, and get
  successively $w = w_1 <_1 \cdots < w_r$ all in the same bubble, until
  we end with $x$ and $y$ in the same $T$-bubble. 

  Still assuming $(T_1,T_2)$ has no global descent,
  we show that $T_1 \wedge T_2$ is a connected poset, or
  $(T_1 \wedge T_2) \vee (T_1 \wedge T_2)^\op = C$.
  Let $x$ be minimal for $T_1$. As we have no global descent, there
  is $x <_1 w$ with $x <_2 w$. So $x$ and $w$ are connected.
  Let $x <_1 y$ be the covering. If $x >_2 y$ then $y <_1 w$ and
  $y <_2 w $ and so $y$ and $w$ are connected. If $x<_2 y$ then
  $x$ and $y$ are connected. In this way we may continue
  and get the $\leq_1$-segment from $x$ to $w$ connected.
  After that, as $w$ is not a global descent, we may continue as above, and
  get a connected $\leq_1$-segment from $w = w_1$ to a large $w_2$.
  Again we continue and eventually get $T_1 \wedge T_2$ a connected poset.

  \medskip
 \noindent {\bf The general situation.}
  From the above, if we have global descents at $x$ and $y$ and no
  global descents in between, the interval $\langle x, y]$ for
  $\leq_1$ gets contained
  in a bubble of $T_1 \vee T_2^\op$. Also any such interval
  gets connected in $T_1 \wedge T_2$. 

  On the other hand, suppose there is 
  a global descent at $x$. Let $D = \{ z \, |\, z \leq_1 x \}$ and
  the complement $U = \{ w \, | \, x <_1 w \}$.
  Then for any $z \in D$ and $w \in U$
  we have $z <_1 w$ and $z >_2 w$ and so $z <_2^\op w$. Hence for
  such $z$ and $w$ we have strictly $z < w$ in $T_1 \vee T_2^\op$.
Thus the ordering relation is strict at global descents.

  Since $z <_1 w$ and $z >_2 w$ for $z \in D$ and $w \in U$, we
  cannot get $z$ and $w$ connected in $T_1 \wedge T_2$, as we cannot
  get any transition from $D$ to $U$.
  Thus the global descents disconnect $T_1 \wedge T_2$.
\end{proof}

\subsection{Refinements of preorders}
\label{subsec:refine}
In Section \ref{sec:pairs} we investigate pairs of preorders $(P,Q)$
and when they give intertwining comonoids.
In  Section \ref{sec:HApairs} 
we improve on the characterization 
by the notion of refinement of partial orders which we now consider.

\begin{definition} \label{def:refine-refine}
A preorder $P$ is a {\it refinement} of $Q$ if:
\begin{itemize}
\item[a.] Each bubble of $P$ is contained in a bubble of $Q$:
  \[ a \circ_P b \Rightarrow a \circ_Q b. \]
\item[b.] If $a,b$ are {\it not} in the {same} $Q$-bubble, then:
  \[ a <_Q b \quad \Leftrightarrow \quad a <_P b. \]
\end{itemize}
Note that $P \preceq Q$.
If the restriction of $P$ to each bubble of $Q$ is a partition order,
then $P$ is a {\it bubble refinement}. This is equivalent to a. and
\begin{itemize}
\item[b'.] For any $a,b$: $a <_Q b \, \Leftrightarrow \, a <_P b$.
\end{itemize}
\end{definition}

\begin{lemma}
  Given a preorder $P$. There is a unique minimal total preorder $T$, such that
  $P$ is a refinement of $T$. We call $T$ the {\it total preorder hull} of
  $P$.
\end{lemma}

\begin{proof}
  If $T_1$ and $T_2$ are two total preorders with $P$ as a refinement, then
  $T_1 \wedge T_2$ is also total: Let $a,b$ in $X$ with $a \leq_1 b$.
We have either $a \leq_2 b$ or $a >_2 b$. 
In the first case $a \leq_{T_1 \wedge T_2} b$. In the latter case,
if $a <_1 b$ and $a >_2 b$, 
  we would have both $a <_P b$ and $a >_P b$, impossible. Hence
  $T_1 \wedge T_2$ is total and it refines $P$. By intersecting all
  total preorders refining $P$ we get the minimal total preorder $T$.
\end{proof}

The following is used in Section \ref{sec:HApairs}. 

\begin{lemma} \label{lem:prelattice-inc}
  Let $B$ be a bubble in the total preorder hull $T$ of $P$. 
  Consider the incomparability relation $\nrel{}$ of $P$ restricted to
  $B$.  If $B$ is
  not a bubble of $P$, the transitive closure of this relation is
  an equivalence relation with
  $B$ as the single class.
\end{lemma}

\begin{proof} If there was an equivalence class with a single
  element, this element had to be comparable in $P$ to everything in $B$,
  and at least strictly comparable to some. But
  then $B$ could not be a bubble in the total preorder hull of $P$.
  Thus every equivalence class has cardinality $\geq 2$. 
  Suppose $E$ and $F$ are distinct classes in $B$ for this equivalence
  relation. Let $x \in E$. Every element of $F$ is comparable with
  $x$. We may then partition $F$ into three classes: $F_{-1}$ those elements
  $< x$, $F_0$ those elements in the same bubble as $x$ and $F_1$, those
  elements $> x$. If $F_0$ is non-empty, an element here would be
  comparable to anything in $F$.
  By definition of the equivalence class of $F$, this would give
  that $F$ only has one element $x$, which we have excluded above. 

  That both $F_{-1}$ and $F_1$ are nonempty is not possible since 
  $y < z$ for any $y \in F_{-1}$ and $z \in F_1$, and then $F$ could
  not be an equivalence class generated by $\nrel{}$.
  Thus everything in $F$ is say $> x$. Similarly picking a $y \in F$,
  everything in $E$ is either $< y$ or $> y$, and the former
  must be the case. The upshot is that $E < F$. In this way
  we may totally order all classes. This contradicts $T$ being
  the total preorder hull with bubble $B$.
  \end{proof}

 \comment{ 
  \begin{lemma} \label{lem:pre-res}
    Let $P_1$ and $P_2$ be preorders on $X$ and $A$ a subset of $X$.
    \begin{itemize}
    \item[a.] The restriction $(P_1 \wedge P_2)_{|A} = P_{1|A} \wedge P_{2|A}$.
    \item[b.] If $A$ is an up-set for both $P_1$ and $P_2$, then
      $(P_1 \vee P_2)_{|A} = P_{1|A} \vee P_{2|A}$.
\end{itemize}
\end{lemma}

\begin{proof}
Part a. is immediate. For part b. suppose $x \leq y$ in $(P_1 \wedge P_2)_{|A}$.
Then we have a sequence
\[ x \leq_1 x_1 \leq_2 x_2 \leq_1 x_3 \cdots \leq_2 x_n = y, \]
where the $x_i$ may possibly not be in $A$. However since $A$ is an up-set
both for $P_1$ and $P_2$, it is clear that every $x_i$ will be in $A$,
and hence $x \leq y$ in $P_{1|A} \wedge P_{2|A}$.
\end{proof}
}

\section{Bimonoids in species over Pre}
\label{sec:bimonoid}
The central idea and tool in constructing various restriction comonoids
is to consider restriction species {\it over the species $\Pre$ of preorders}
on a set $X$. We are interested in pairs of restriction comonoids
$\Delta^1, \Delta^2$ such that when dualizing $\Delta^2$ we get
a bimonoid. We develop the criteria needed to be checked for this.

\subsection{Species over preorders}
By the previous section we get a restriction species
\[ \Pre : \setii \pil \set, \quad X \mapsto \Pre(X). \]

\begin{definition} \label{def:bimonoid-rsp}
  Let $\sS : \setii \pil \set$ be a restriction species,
  and $\pi : \sS \pil \Pre$ a natural transformation of species.
  We do not require it to be a natural transformation of restriction
  species but require the following weaker condition.
  For each injection $i : Y \inpil X$ consider the diagram:
  \[ \xymatrix{\sS[X] \ar[r]^{} \ar[d]^{\pi[X]}  
      &\sS[Y] \ar[d]^{\pi[Y]}\\
      \Pre[X] \ar[r]^{} & \Pre[Y]. }
  \]
$\sS$ is a {\it restriction species over preorders} if:
\begin{itemize}
\item[1.] For each $s \in \sS[X]$ we have
  $\pi(s_{|Y}) \preceq \pi(s)_{|Y}$,
\item[2.] When $(U,V)$ is a cut for $\pi(s)$, the above is equality:
  \begin{itemize}
  \item[a.] $\pi(s_{|U}) = \pi(s)_{|U}$,
  \item[b.] $\pi(s_{|V}) = \pi(s)_{|V}$,
  \end{itemize}
\end{itemize}
\end{definition}

In particular for each set $X$ we have a set map $\sS[X] \pil \Pre[X]$.
The preorder we get after restricting $s$ will in general be finer
than the preorder we get from $s$ and then restricting.

We now consider the restriction species $(\sS,\pi)$ over preorders to
be a species over $\setmm$. It induces a natural 
{\it restriction coproduct}:
\begin{equation} \label{eq:bimonoid-resco}
  \Delta : \sS[X] \pil \sS[A] \times \sS[B], \quad
  s \mapsto
  \begin{cases} (s_A, s_B), & (A,B) \text{ is cut for } \pi(s) \\
    \nil, & (A,B) \text{ not cut for } \pi(s) 
  \end{cases}.
\end{equation}

Due to condition 2 in Definition \ref{def:bimonoid-rsp}, this coproduct
is associative. Furthermore by Remark \ref{rem:setcat-gradeo}
we have a counit $\epsilon$. Thus $(\sS, \Delta, \epsilon)$ becomes
a comonoid.

\comment{
So we have commutative diagrams:

  \[ \xymatrix{\sS[X] \ar[r]^{\Delta^S_{A,B}} \ar[d]^{\pi[X]}  
      & \sS[A] \times \sS[B] \ar[d]^{\pi[A] \times \pi[B]}\\
      \Pre[X] \ar[r]^{\Delta^{cut}_{A,B}} & \Pre[A] \times \Pre[B], } \quad
   \xymatrix{\sS[X] \ar[r]^{\epsilon} \ar[d]^{\pi[X]}  &
      \unit[X] \ar[d]^{=}\\
      \Pre[X] \ar[r]^{\epsilon} & \unit[X]}.    
      \]
}
  We now assume the restriction species $\sS$ has two
  structures $\pi_1$ and $\pi_2$ as restriction species over $\Pre$.
  We then get two restriction comonoids $\Delta^1$ and $\Delta^2$.
  We are interested in when these two comonoids fulfill
  Proposition \ref{pro:setcat-DeMu}.

  \begin{definition}
\label{def:species-inter}
    Two restriction comonoids in species $\Delta^1$ and $\Delta^2$
    are {\it intertwined} if for {\it every} decomposition
$X = A \scup B \scup C \scup D$, the diagram
\begin{equation} \label{eq:species-pb}
  \xymatrix{ \sS[X] \ar[r]^{\Delta^2_{A C,B D} \hskip 3mm} \ar[d]_{\Delta^1_{A B, C D}}
      & \raisebox{-3.5mm}{\STvb{\sS\left [ \begin{matrix} A \\ C \end{matrix} \right ]}
        {\sS\left [\begin{matrix} B \\ D \end{matrix} \right ]}
          \ar[d]^{\Delta^1_{B,D}}_{ \Delta^1_{A,C}}} \\
        \raisebox{-5mm}{\SThb{\sS[A \, B]}{\sS[C \, D]}
          \ar[r]^{\Delta^2_{A,B}}_{\Delta^2_{C,D}}}
          & \raisebox{-6.5mm}{\ABCDb{\sS[A]}{\sS[B]}{\sS[C]}{\sS[D]}}}.        
\end{equation}
is a partial pull-back diagram.
\end{definition}

  \begin{lemma} Let a restriction species
    $\sS$ have two structures $\pi_1, \pi_2$ as
    restrictions species over preorders. 
    Let $s \in \sS[X]$ and suppose $(D^1,U^1)$ is a cut
  for $\pi_1(s)$ and $(D^2,U^2)$ a cut for $\pi_2(s)$.
  Let
  \[ A = D^1 \cap D^2, \quad B = D^1 \cap U^2, \quad C = U^1 \cap D^2,
    \quad D = U^1 \cap U^2. \]
Consider the restrictions $s_{AC}, s_{AB}, s_{BD}, s_{CD}$. Then:
\begin{itemize}
\item[1.] $(A,C)$ a cut for $\pi_1(s_{AC})$, $(B,D)$ a cut for $\pi_1(s_{BD})$
  \item[2.] $(A,B)$ a cut for $\pi_2(s_{AB})$, $(C,D)$ a cut for $\pi_2(s_{CD})$.
  \end{itemize}
\end{lemma}

\begin{proof} We show this for $(A,C)$. The others are similar.
We have $\pi_1(s_{AC}) \preceq \pi_1(s)_{|AC}$. Also $(D^1, U^1)
= (A \scup B, C \scup D)$ is a cut for $\pi_1(s)$. Then $(A,C)$
is a cut for $\pi_1(s)_{|AC}$ and so also for $\pi_1(s_{AC})$.
\end{proof}

\begin{corollary} \label{cor:species-pb}
  The diagram \eqref{eq:species-pb} fulfills
  part 1 in Definition \ref{def:setcat-pb} for a partial pullback diagram.
\end{corollary}

\subsection{Are two restriction comonoids intertwined?}
\label{subsec:intertwined-species}

We can now summarize the procedure to check if a restriction species
$\sS : \setii \pil \set$  with two natural transformations $\pi_1$ and $\pi_2$
to preorders gives rise to intertwined restriction comonoids.

\begin{itemize}
\item {\it Species over preorders:} We verify that $\pi_1$ and $\pi_2$ make
  $\sS$ a species over preorders. This amounts to check:
  \begin{itemize}
    \item[1.] For every injection $Y \hookrightarrow X$ and $s \in \sS[X]$:
  \[ \pi_1(s_{|Y}) \preceq \pi_1(s)_{|Y}, \quad
    \pi_2(s_{|Y}) \preceq \pi_2(s)_{|Y}. \]
  By Corollary \ref{cor:species-pb} above, this verifies part 1 of
  Definition \ref{def:setcat-pb}.
\item[2.] For $i = 1,2$, when $(U,V)$ is a cut for $\pi_i$, then
  \[ \pi_i(s_{|U}) = \pi_i(s)_{|U}, \quad \pi_i(s_{|V}) = \pi_i(s)_{|V}. \]
  This verifies that $(\sS,\Delta_i, \epsilon_i)$ are comonoids.
\end{itemize}

\item {\it Extension:}  Given a diagram:
  \begin{equation} \label{eq:species-uniq1}
  \xymatrix{ \sS[X] \ar@{-->}[r]^{\Delta^2}
   \ar@{-->}[d]_{\Delta^1}
   & \raisebox{-3.5mm}{\STvb{\sS\left [
         \begin{matrix} A \\ C \end{matrix} \right ]}
        {\sS\left [\begin{matrix} B \\ D \end{matrix} \right ]}
        \ar[d]^{\Delta^1}_{\Delta^1}} \\
        \raisebox{-5mm}{\SThb{\sS[A \, B]}{\sS[C \, D]}
          \ar[r]^{\Delta^2}_{\Delta^2}}
          & \raisebox{-6.5mm}{\ABCDb{\sS[A]}{\sS[B]}{\sS[C]}{\sS[D]}}
        },
        \quad
       \xymatrix{ s
    \ar@{-->}[d] \ar@{-->}[r] &
    \raisebox{-6mm}{\STv{s_{AC}\, \, \,}
      {\, \, \, s_{BD}}} \ar[d]^{\Delta^1}_{\Delta^1}\\
    \raisebox{-5mm}{\STh{s_{AB}}{s_{CD}}} \ar[r]^{\Delta^2}_{\Delta^2}
    & \raisebox{-5mm}{\ABCD{s_A}{s_B}{s_C}{s_D}} }
\end{equation}
where:
\begin{itemize}
\item[1.] $(A,C)$ a cut for $\pi_1(s_{AC})$, $(B,D)$ a cut for $\pi_1(s_{BD})$
  \item[2.] $(A,B)$ a cut for $\pi_2(s_{AB})$, $(C,D)$ a cut for $\pi_2(s_{CD})$.
  \end{itemize}
  Verify for the right diagram that there
  is exactly one way to complete this to an element
    $s \in \sS[X]$ restricting to $(s_{AC}, s_{BD})$ and to
    $(s_{AB}, s_{CD})$. (Alternatively there could be several such $s$,
    but then consider the next step.)

  \item {\it Cuts:} Verify for this $s$ that $(A \sqcup C, B \sqcup D)$
  is a cut for $\pi_2(s)$ and $(A \sqcup B, C \sqcup D)$ is a cut for
  $\pi_1(s)$. (If there are several $s$ verify that there is a unique $s$
  for which these are cuts.)
  This verifies part 2 of Definition \ref{def:setcat-pb}.
\end{itemize}

This gives that the diagram \eqref{eq:species-pb}
is a partial pullback diagram.
Using Proposition \ref{pro:setcat-DeMu}, by dualizing $(\Delta^2, \epsilon^2)$
we get a bimonoid species 
$\sB^1$,  and by
dualizing $(\Delta^1, \epsilon^1)$ we get at bimonoid species
$\sB^2$:
\[ \sB^1 = (\sS, \Delta^1, \mu_2, \epsilon^1, \iota_2), \quad
  \sB^2 = (\sS, \mu_1, \Delta^2, \iota_1, \epsilon^2). \]

\section{Sub-bimonoids by avoidance}
\label{sec:avoid}

Many Hopf algebras arise as sub-Hopf algebras or quotient Hopf algebra
of larger Hopf algebras, by requiring the basis elements to avoid
certain configurations. Both
the Connes-Kreimer and symmetric functions Hopf algebras arise
in this way, and also the Loday-Ronco Hopf algebra and quasi-symmetric
functions.

This raises the question of searching for ``master Hopf algebras'',
those who do not apparently come from larger Hopf algebras by avoiding
certain configurations.
We shall see some such large Hopf algebras in the latter sections. But
let us here clarify the notions.


\begin{definition}
 Let $\sA$ be a subspecies of the restriction species $\sS$ in $\set$. (We do
not assume $\sA$ is also a restriction species.) 

An $s \in \sS[X]$ {\it avoids} $\sA$ if there is {\it no injection}
$Y \overset{u}{\inpil} X$ with $a = s_Y$ in $\sA[Y]$.
Otherwise $s$ {\it has an $\sA$-part}. 

For each finite set $X$ let $\sSA[X]$ be the
subset of $\sS[X]$ of $\sA$-avoiding elements.
This is the {\it $A$-avoiding subspecies} of $\sS$.
\end{definition}

We now consider $\sS$ as a species in $\setmm$, and let 
$\Delta$ be a restriction comonoid on $\sS$. For 
\[ \Delta : \sS[X] \pil \sS[U] \times \sS[V], \]
if $\Delta(s) = (s_U,s_V)$ is nonzero and 
if $s$ is $\sA$-avoiding, both
$s_U$ and $s_V$ will be $A$-avoiding. Thus $\sSA$ is a sub-comonoid
species of $\sS$.

\begin{definition}
  The coproduct $\Delta $ is $\sA$-{\it irreducible},
  if whenever $s \in \sS[X]$ has an $\sA$-part, and $X = U \scup V$  with
  $\Delta(s) = (s_U,s_V)$ {\it non-zero},
  then either $s_U$ or $s_V$ has an $\sA$-part.
\end{definition}


Since $\sSA$ is a comonoid subspecies of $\sS$ we have commutative
diagrams
\[ \xymatrix{ \sSA[X] \ar[r] \ar[d] & \sSA[U] \times \sSA[V]
    \ar[d] \\
    \sS[X] \ar[r] & \sS[U] \times \sS[V] }.
\]

\begin{lemma}
  When $\Delta$ is $\sA$-irreducible, $\sSA$ is in addition a quotient
  comonoid of $\sS$ for $\Delta$.
  Thus $\sSA$ is a split sub-comonoid of $\sS$.
\end{lemma}

\begin{proof}
 Define the quotient map $q : \sS[X] \pil \sSA[X]$ by  $q(s) = s$ if
  $s$ is $A$-avoiding, and $q(s) = \nil $ if not.  
  Consider the diagram (see Subsection \ref{subsec:setmm}
  for defining the map on products):
\[ \xymatrix{\sS[X] \ar[r]^-{\Delta} \ar[d] & \sS[U] \times \sS[V] \ar[d] \\
    \sSA[X] \ar[r] & 
      \sSA[U] \times \sSA[V]. } \]
  If $s \in \sS[X]$ has an $\sA$-part, and $\Delta(s) = (s_U,s_V)$ is
  nonzero, then
  $s_U$ or $s_V$ has an $\sA$-part (by the $\sA$-irreducibility of $\Delta)$.
  Thus both images in the lower row are zero. This shows the diagram
  is commutative.
\end{proof}


Suppose now $\sS$ has two intertwining restriction comonoid coproducts
$\Delta^1$ and $\Delta^2$.
These coproducts restrict to coproducts on $\sSA$, but may no longer
be intertwining.
Recall that  $\sB^1$ is the bimonoid species $(\sS, \Delta^1, \mu_2, \epsilon^1,
\iota_2)$ derived from $\sS$, and $\sB^2$ the bimonoid species
$(\sS, \mu_1, \Delta^2, \iota_1, \epsilon^2)$.
Correspondingly we may possibly get bimonoid species $\sBA^1$ and $\sBA^2$.

\begin{proposition} \label{pro:avoid-subquot} {}
 \noindent a.  Suppose $\Delta^1$ or $\Delta^2$
  is $\sA$-irreducible. Then $\Delta^1$ and $\Delta^2$
  restricted to $\sSA$ are intertwined, and so we
  get dual bimonoid species $\sBA^1$ and $\sBA^2$. 

\noindent b. If $\Delta^2$ is $\sA$-irreducible, then $\sBA^1$ is
  a sub-bimonoid species of $\sB^1$ and $\sBA^2$ a quotient bimonoid species
  of $\sB^2$.

\noindent c. If $\Delta^1$ is $\sA$-irreducible, then $\sBA^2$ is
  a sub-bimonoid species of $\sB^2$ and $\sBA^1$ a quotient bimonoid species
  of $\sB^1$.
\end{proposition}

\begin{proof}
a. Consider the diagram 
 \begin{equation} \label{eq:species-uniq2}
  \xymatrix{ \sS[X] \ar@{-->}[r]^-{\Delta^2}
   \ar@{-->}[d]_{\Delta^1}
   & \raisebox{-3.5mm}{\STvb{\hskip -3mm \sSA\left [
         \begin{matrix} A \\ C \end{matrix} \right ]}
        {\hskip 3.5mm\sSA\left [\begin{matrix} B \\ D \end{matrix} \right ]}
        \ar[d]^-{\Delta^2}_-{\Delta^2}} \\
        \raisebox{-5mm}{\SThb{\sSA[A \, B]}{\sSA[C \, D]}
          \ar[r]^-{\Delta^1}_-{\Delta^1}}
        & \raisebox{-6.5mm}{\ABCDb{\hskip -3mm \sSA[A]}
          {\hskip 4mm \sSA[B]}{\hskip -3mm \sSA[C]}{\hskip 4mm \sSA[D]}}.
        },
        \quad
       \xymatrix{ s
    \ar@{-->}[d] \ar@{-->}[r] &
    \raisebox{-6mm}{\STv{\hskip -2mm s_{AC}}
      {\hskip 2 mm s_{BD}}} \ar[d]^-{\Delta^2}_-{\Delta^2}\\
    \raisebox{-5mm}{\STh{s_{AB}}{s_{CD}}} \ar[r]^-{\Delta^1}_-{\Delta^1}
    & \raisebox{-5mm}{\ABCD{s_A}{s_B}{s_C}{s_D}}. }
\end{equation}
If $\Delta^1$ is $\sA$-irreducible, any extension $s$ (in the right
square) will be $\sA$-avoiding.
Similarly if $\Delta^2$ is $\sA$-irreducible.
So $\Delta^1$ and $\Delta^2$ are intertwined for $\sSA$, and we
get bimonoid species $\sBA^1$ and $\sBA^2$. 

\medskip
\noindent b. If $\Delta^2$ is $\sA$-irreducible, $\sSA$ is a quotient comonoid species
for $\Delta^2$, and so $\sBA^1$ is a submonoid species of $\sB^1$ for
the multiplication $\mu_2$. It
is also a comonoid subspecies for $\Delta^1$, and so a sub-bimonoid species
of $\sB^1$. Dualizing we get $\sBA^2$ as a quotient bimonoid species
of $\sB^2$. Part c. is similar.
\end{proof}

 \medskip \medskip
{\noindent {\large \bf Part II: Constructions of Hopf algebras.}

\section{Hopf algebras from restriction species over
  preorders}
\label{sec:respre}
We show how the following Hopf algebras (HA) come from restriction
species over preorders:
\begin{multicols}{2}
\begin{itemize}
\item Commutative polynomial ring
\item Tensor algebra
\item Schmitt's HA of graphs
\item The HA of posets
\item The Connes-Kreimer HA
\item The HA of symmetric functions
\end{itemize}
\end{multicols}
First we recall the Fock functor.

\subsection{The Fock functor}
Let
\[ \sS : \setx \pil \setmm \]
be a species. For $X = [n] = \{1,2, \ldots, n\}$
write $\sS[n]$ for $\sS[X]$. 
Let $\sS[n]_{S_n}$ be the orbits of $\sS[n]$, the {\it coinvariants},
under the action of
the symmetric groups $S_n$.
The (bosonic) Fock functor of $\sS$ is 

\[ \clK(\sS) = \oplus_{n \in \hele} \kk \sS[n]_{S_n}. \]
In our cases we assume $\sS[\emptyset] = \{*\}$, and $\epsilon$
the natural map, Definition \ref{def:setcat-resco}.
Then if $\sS$ is a bimonoid species
over $\setmm$ it becomes a Hopf monoid species over $\setz$. 
The Fock functor applied to $\sS$ then becomes a Hopf algebra,
\cite[Section 15]{AgMa}.

\medskip
When considering species $\sS$ over vector spaces, one may also
consider the {\it invariants} $\sS[n]^{S_n}$, and this gives
the contragredient Fock functor \cite[Chap.15.1]{AgMa}:
\[ \clK^\vee(\sS) = \oplus_{n \in \hele} \kk \sS[n]^{S_n}. \]
In our situation with a species $\sS$ over $\setmm$,
$\clK(\sS)$ and $\clK^\vee(\sS)$ are isomorphic as
graded vector spaces, but may not be isomorphic
as Hopf algebras.
However, in general these Hopf algebras are isomorphic when the
characteristic of the field $\kk$ is zero.

\medskip
When the species $\sS : \setx \pil \setmm$
gives a bimonoid $\sB = (\sS,\Delta, \mu,
\epsilon, \iota)$, we may dualize all maps and get a dual Hopf
species $\sB^* = (\sS, D\mu, D\Delta, D\iota, D\epsilon)$.
When $H$ is a {\it graded} Hopf algebra with finite dimensional graded pieces,
we also get a graded dual Hopf algebra $H^*$. 
By \cite[Thm.5.13.]{AgMa} we have
\[ \clK^\vee(\sB^*) = \clK(\sB)^*. \]

\subsection{The symmetric algebra} \label{sub:fock-pol}
Let $F$ be a fixed set of variables (or colors) $\{x_1, x_2, \ldots, x_f\}$. 
Let $\sS[X]$ be the set of functions $s : X \pil F$.
Note that the coinvariants $\sS[n]_{S_n}$ correspond to monomials
$x_1^{a_1}x_2^{a_2} \cdots x_f^{a_f}$ of degree $n$. 
For $s \in \sS[X]$
let both the first
and second preorder be the discrete preorder on $X$ (see Subsection
\ref{subsec:preorders}):
\[ \pi_1(s) = \pi_2(s) = D(X). \]
For $Y \sus X$ there are natural restriction maps $\sS[X] \pil \sS[Y]$.
This gives coproducts $\Delta^1 = \Delta^2$, both equal to the natural
restriction maps:
\[ \sS[X] \pil \sS[A] \times \sS[B], \quad  s \mapsto (s_A, s_B) \]
for any decomposition $X = A \sqcup B$.
Consider diagram \eqref{eq:species-pb}. 

\medskip
\noindent {\it Species over preorders:} This is clear.

\medskip
\noindent {\it Extension:}
Given the diagram \eqref{eq:species-uniq1},
it is clear that $s_{AB}, s_{AC}, s_{CD}$ and $s_{BD}$ glue together to give
a unique map $s: X \pil F$.

\medskip
\noindent {\it Cuts:} Clearly
$(A \sqcup B, C \sqcup D)$ and $(A \sqcup C, B \sqcup D)$ are cuts
for $\pi_1(s) = \pi_2(s) = D(X)$.

\medskip
Thus we have two intertwining comonoids of species. Dualizing
we get a bimonoid species $\sB^1 = (\sS,\Delta^1, \mu_2, \epsilon^1, \iota_2)$. 
The Fock functor $\clK(\sB^1)$ becomes the symmetric (polynomial) algebra
$k[x_1, \ldots, x_f]$ with its natural Hopf algebra structure. The
element $s \in \sS[X]$ is sent to the monomial $\prod_{x \in X} s(x)$
in this polynomial ring.

\begin{remark}
  The contragredient Fock functor $\clK^\vee$ sends $\sB^1$
  (or the isomorphic dual bimonoid $\sB^2 = (\sS, \mu_1, \Delta^2,
  \iota_1, \epsilon^2)$ )
  to the divided powers algebra, which is the dual of the
  symmetric Hopf algebra.
\end{remark}

\begin{remark} \label{rem:fock-DC}
  If we had chosen the $\pi_1(s) = D$ and $\pi_2(s) = C$ the
  coarse preorder, this would not give a Hopf species.
  The only possible cuts for $C$ on $X$ are
  $(X, \emptyset)$ and $(\emptyset,X)$.
  Let $X$ and $Y$ be non-empty sets. Consider the diagram:
\begin{equation*} \label{eq:eks-pbxy}
  \xymatrix{ \sS[X \sqcup Y] \ar@{-->}[r]^{\Delta^2}
  \ar@{-->}[d]_{\Delta^1}
& \raisebox{-3.5mm}{\STvb{\sS[X]}
        {\sS[Y]}
          \ar[d]^{\Delta^1}_{\Delta^1}} \\
        \raisebox{-5mm}{\SThb{\sS[X]}{\sS[Y]}
          \ar[r]^{\Delta^2}_{\Delta^2}}
        & \raisebox{-6.5mm}{\ABCDb{\sS[X]}{\sS[\emptyset]}
          {\sS[\emptyset]}{\sS[Y]}}.
      },     
\end{equation*}
Given elements $s_X \in \sS[X]$ and $s_Y \in \sS[Y]$ we can complete
it to a unique element $s \in \sS[X \cup Y]$ restricting to $s_X$ and $s_Y$.
But $(X,Y)$ will not be a cut for $\pi_2(s) = C$.
\end{remark}

\subsection{The tensor algebra}
Let $\sT[X]$ consists of pairs $(s,T)$ where $s : X \pil F$ is a function,
and $T$ is a total order on $X$.
The coinvariants $\sT[n]_{S_n}$ then correspond to words in $F$ of length $n$. 
Let the associated preorders be given by
\[ \pi_1(s,T) = D(X), \quad \pi_2(s,T) = T. \]

For a cut $(U,V)$ for the total order $T$, the restriction coproduct
\[ \Delta^2 : \sT[X] \pil \sT[U] \times \sT[V], \quad \text{sends }
(s,T) \mapsto (s_U, T_{|U}) \times (s_V, T_{|V}).\]

Similarly for any decomposition $X  = A \sqcup B$ we get the restriction
coproduct
\[ \Delta^1 : \sT[X] \pil \sT[A] \times \sT[B]. \]

\medskip
\noindent {\it Species over preorders:}
For $Y \sus X$:
\[ \pi_2(s,T)_{|Y} = T_{|Y} = \pi_2(s_Y,T_{|Y}), \]
and the similar holds obviously for $\pi_1$.

\medskip
\noindent {\it Extension:}
Given the diagram \eqref{eq:species-uniq1}, it
is clear that the maps $s_{AB}, s_{AC}, s_{CD}$ and $s_{BD}$ glue together to give
a unique map $s: X \pil F$. Furthermore define the total order $T$ by
its restrictions to $A \sqcup C$ and $B \sqcup D$ being given by
respectively $T_{AC}$ and $T_{BD}$ and if $x \in A \sqcup C$ and $y
\in B \sqcup D$ then $x <_T y$. This is the unique possible order
if $A \sqcup C, B \sqcup D$ is a cut for $T$.

\medskip
\noindent {\it  Cuts:}
It is clear that 
$(A \sqcup B, C \sqcup D)$ and $(A \sqcup C, B \sqcup D)$ are cuts
for $\pi_1(s,T) = D(X)$ and for $\pi_2(s,T) = T$, respectively.

\medskip
We get a bimonoid species $\sB^1 = (\sT, \Delta^1,\mu_2,\epsilon^1, \iota_2)$
whose associated Hopf algebra by the Fock functor $\clK$ is the tensor
algebra $k\langle x_1, \ldots, x_f \rangle$. The dual bimonoid
species $\sB^2 = (\sT, \mu_1, \Delta^2, \iota_1, \epsilon^2)$ gives
the shuffle Hopf algebra by the Fock functor $\clK^\vee$.


\subsection{Hopf algebra of graphs}
Let $\sG[X]$ be the (simple) graphs with vertex set $X$.
For a graph $G$ we let $\pi_1(G) = D$ the discrete order and
$\pi_2(G)$ the partition order whose bubbles are the connected
components of $G$. So $x \leq_2 y$ if $x$ and $y$
are connected by a path in $G$. This gives $\sG[X]$ two structures
of species over preorders.
We may note that for $Y \sus X$ then
$\pi_2(G_{|Y})$ in general
is a finer partition order than the restriction of $\pi_2(G)_{|Y}$,
since after restricting to a subset $Y$ we may get more components
than before we restricted.
The coproduct
\[\Delta^1 : \sG[X] \pil \sG[U] \times \sG[V], \quad
G \mapsto (G_{|U}, G_{|V}) \]
is simply restriction on each factor for any decomposition
$X = U \scup V$. The coproduct for $\Delta^2$ is
similar whenever $(U,V)$ is a cut for $\pi_2(G)$.
It is straightforward to verify that these coproducts are intertwined.

Applying the Fock functor to the bimonoid species
$\sB^1 = (\sG,\Delta^1, \mu_2, \epsilon^1, \iota_2)$ we get
Schmitt's Hopf algebra of graphs \cite{Sch-Inc}.

\subsection{Hopf algebras of posets or preorders}
Let $\sP[X]$ be the set of all posets (alternatively we may consider
all preorders) on $X$.
For a poset $P$ let $\pi_1(P) = P$ and $\pi_2(P) = P^{\bullet}$, the
partition order whose bubbles are the connected components of $P$.
It gives the restriction coproduct
\[ \Delta^1 : \sP[X] \pil \sP[U] \times \sP[V], \quad P \mapsto
  (P_{|U}, P_{|V}), \]
when $(U,V)$ is a cut for $\pi_1(P) = P$, and similarly for $\Delta^2$. 
The corresponding restriction comonoids are easily checked to be
intertwined. 

Applying the Fock functor to the bimonoid species
$\sB^1 = (\sP, \Delta^1, \mu_2, \epsilon^1, \iota_2)$ we get
Schmitt's Hopf algebra of posets, \cite{Sch-Inc}, or more
generally the Hopf algebra of finite topologies \cite{FaFoMa}.

Note that if we consider any class of posets closed under cuts
and disjoint unions
we get a Hopf sub-algebra of the Hopf algebra of posets.

\begin{example} {\bf Connes-Kreimer.}
The class of posets
avoiding the cherry poset, are the posets whose Hasse diagram is
a forest. Then $\Delta^2$ is irreducible for this class, and
we get the Connes-Kreimer Hopf algebra as a sub-Hopf algebra.
\end{example}

\begin{example}{\bf Symmetric functions.}
  The class of posets avoiding the cherry poset and its opposite,
  the $V$-poset,
  are those whose connected components are total orders.
  Again $\Delta^2$ is irreducible for this class, and 
  we get the Hopf algebra of symmetric functions as a sub-Hopf algebra.
\end{example}

\section{Many Hopf algebras of permutations}
\label{sec:perm}

In this section the Malvenuto-Reutenauer (MR) Hopf algebra \cite{MR}
is the master Hopf algebra. We give a wealth of quotient Hopf algebras
of the MR-algebra, Theorem \ref{thm:perm-A},
the Loday-Ronco Hopf algebras being a special case.
Taking an {\it arbitrary} family of permutations without global
descents the permutations avoiding the patterns of this family form
a quotient Hopf algebra. The Loday-Ronco Hopf algebra is the special
case when the family is the single permutation $2 1 3$. 

We give {\it two distinct ways} to get the MR-Hopf algebra
from intertwining restriction comonoids. The first corresponds
to the original way of defining the MR-Hopf algebra in \cite{MR}, sometimes
called the $F$-basis. 
The second corresponds to the $M$-basis in \cite{AgSo} and
to the plane partition basis in \cite{Fo-PP}. This second basis turns out
to be the effective one for considering sub- and quotient Hopf
algebras of avoidance,
Theorem \ref{thm:perm-A}.

Permutation patterns from restrictions species are also
considered in \cite{Pe}, inspired by \cite{Va}.
The viewpoint there is not avoidance, but grouping permutations together
which possess a certain pattern. 

\subsection{Malvenuto-Reutenauer} \label{sub:eks-MR1}
Let $\sS[X]$ be the set of pairs $s = (T_1, T_2)$ where $T_1$ and $T_2$ are
total orders on $X$. Identifying $T_1$ with $\{ 1 < 2 < \cdots < n\}$, this
corresponds to a permutation of this latter set.

Let $\pi_1(T_1,T_2) = T_1$ and $\pi(T_1, T_2) = T_2$.
The coproduct map is:
\begin{align*} \Delta^1 : \sS[X] & \lpil  \sS[U] \times \sS[V] \\
  (T_1,T_2) & \mapsto \begin{cases}  (T_{1U}, T_{2U}; T_{1V},T_{2V}) &
    (U,V) \text{ is a cut for } T_1 \\
    \nil & \text{ otherwise}
  \end{cases}
\end{align*}
and similarly for $\Delta^2$.

\medskip
\noindent {\it Species over preorders:}
This is clear.

\medskip
\noindent {\it Extension:}
In the diagram \eqref{eq:species-uniq1}, to extend to an element
$s = (T_1, T_2)$ of $\sS[X]$, for $(A \sqcup C, B \sqcup D)$ to be
a cut we must have $x <_{T_2} y$ for $x \in A \sqcup C$ and $y \in B \sqcup D$.
Otherwise the order on $T_2$ is given by those in $T_{2|AC}$ and $T_{2|BD}$.
Similarly we extend to $T_1$ such that $(A \sqcup B, C \sqcup D)$
is a cut for $T_1$.

\medskip
\noindent {\it Cuts:} Clearly
$(A \sqcup B, C \sqcup D)$ and $(A \sqcup C, B \sqcup D)$ are cuts
for $T_1$ and $T_2$ respectively.

\medskip
Let $\sB^1$ be the bimonoid species $(\sS, \Delta^1, \mu_2, \epsilon^1, \iota_2)$.
Applying the Fock functor $\clK(\sB^1)$ we get the Malvenuto-Reutenauer (MR)
Hopf algebra. 

If $(T_1, T_2)$ corresponds to the permutation $\sigma$,
then $(T_1,T_2)$ maps to $F_\sigma$ in the notation of \cite{AgSo}.
The coproduct in the MR-Hopf algebra $\clK(\sB^1)$
is then given by deconcatenation and standardization:
\[ \Delta(F_{3124}) = 1 \te F_{3124}  + F_{1} \te F_{123} + F_{21} \te F_{12}
  + F_{312} \te F_1 + F_{3124} \te 1, \]
and the product is given by shifting the second factor and shuffling
\begin{align*}  F_{12} \cdot F_{312} = & F_{12534} + F_{15234} + F_{15324}
  + F_{15342} + F_{51234} + F_{51324} \\
  = & + F_{51342} + F_{53124} + F_{53142} + F_{53412}
\end{align*}

\subsection{Malvenuto-Reutenauer II} \label{sub:eks-MR2}
Consider again
$\sS[X]$ consisting of pairs $s = (T_1,T_2)$ of total orders on $X$.
In \cite{Fo-PP}, L.Foissy constructs a bijection between permutations in $S_n$
and plane posets with $n$ elements. The corresponding plane
poset has two orders, the vertical order $\leq_v$ which is given
by $T_1 \wedge T_2$ and the horizontal order $\leq_h$ given by
$T_1 \wedge T_2^{\op}$.

These two will not work as projection preorders for $\sS[X]$, but it is
closely related to the following which does work.
Let $\pi_1(s) = T_1 \vee T_2^{\op}$ and $\pi_2(s) = T_1 \wedge T_2$.
Recall from Section \ref{sec:pre} that the former is a
total preorder corresponding to the global descents of the permutation
associated to $(T_1,T_2)$. 
When $(U,V)$ is a cut for $T_1 \vee T_2^{\op}$, the coproduct 
\[ \Delta^1 : \sS[X] \pil \sS[U] \times \sS[V],
\quad (T_1,T_2) \mapsto (T_{1|U}, T_{2|U};T_{1|V}, T_{2|V}). \]
The coproduct $\Delta^2$ is defined in the same way when $(U,V)$ is
a cut for $T_1 \wedge T_2$. 

\medskip
\noindent{\it Species over preorders:}
Lemma \ref{lem:pre-resY} shows this.

\medskip
\noindent{\it Extension:}
Consider the diagram \eqref{eq:species-uniq1}.
We determine both $T_1$ and $T_2$ from the pair $s_{AB}$ and $s_{CD}$.
Since $T_1 \preceq T_1 \vee T_2^{\op}$, the pair $(A \sqcup B, C \sqcup D)$
is a cut for $T_1$ also. Hence for $x \in A \sqcup B$ and $y \in C \sqcup D$
we must have $x <_{T_1} y$, and on $A \sqcup B$ and $C \sqcup D$ we
determine $T_1$ from the restrictions $T_{1|AB}$ and $T_{1|CD}$.
Similarly $T_2$ may be determined from the cut $(A \sqcup B, C \sqcup D)$.

\medskip
\noindent {\it Cuts:}
By construction 
$(A \sqcup B, C \sqcup D)$ is a cut
for $T_1 \vee T_2^{\op}$. Let us argue that
$(A \sqcup C, B \sqcup D)$ is a cut for $T_1 \wedge T_2$.
Let  $x \in A \sqcup C$ and $y \in B \sqcup D$. What could go
wrong is that $y <_{T_1 \wedge T_2} x$. Looking at $T_1$
this cannot happen if both are
in $A \sqcup B$ (then $x \in A$ and $y \in B$) or both in $C \sqcup D$
(then $x \in C$ and $y \in D$).
But since $A \sqcup B, C \sqcup D$ is a cut for $T_1$ and for $T_2^\op$,
if $x \in A$ and $y \in D$, then $x <_{T_1} y$, and if $x \in C$ and
$y \in B$ then $x <_{T_2} y$. Hence we cannot have both
$y <_{T_1} x$ and $y <_{T_2} x$, and so
$(A \sqcup C, B \sqcup D)$ is a cut for $T_1 \wedge T_2$.

\medskip
Now let $\sB^1$ be the bimonoid species $(\sS,\Delta^1, \mu_2, \epsilon^1, \mu_2)$.
Then $\clK(\sB^1)$ is a Hopf algebra, and actually it is again
the Malvenuto-Reutenauer Hopf algebra. This time, when the pair $(T_1,T_2)$
corresponds to the permutation $\sigma $, the pair maps to
the basis element $M_{\sigma}$
in the notation of \cite{AgSo}. The coproduct is then defined
in terms of global descents:
\[ \Delta^1(M_\sigma) = \sum_{} M_{\tau}
  \te M_{\rho}, \]
the sum over all $1 \leq k \leq n-1$ where $\sigma$ has a global
descent at position $k$,
and $\tau$ is the standardization of $\sigma$ restricted
to $\{1,2,\cdots, k\}$ and $\rho$ the standardization of
$\sigma$ to $\{k+1, \cdots, n\}$.

The bases $M_\sigma $ and $F_\sigma $
are related by
\[ F_\sigma = \sum_{\tau \geq \sigma} M_\tau. \]
That this gives an isomorphism between $\clK(\sB^1)$ and the
Malvenuto-Reutenauer algebra constructed in Subsection \ref{sub:eks-MR1}
is Corollary 5.11 in \cite{Fo-PP}.

Let  $\sB^2$ be the bimonoid species $(\sS,\mu_1, \Delta^2, \iota_1,
\epsilon^2)$, the dual of $\sB^1$.
Then $\clK^\vee(\sB^2) $ is the dual Hopf algebra of
$\clK(\sB^1)$. In this case $(T_1,T_2)$ maps to $M_\sigma^*$.
In $\clK^\vee(\sB^2)$ of Subsection \ref{sub:eks-MR1} $(T_1,T_2)$  maps to
$F_\sigma^*$. These bases are now related
by $M_\sigma^* = \sum_{\tau \leq \sigma} F_\tau^*$ as can be checked by
the coproducts $\Delta^2$.

\subsection{The Loday-Ronco Hopf algebra and
  Foissy Hopf algebra   of planar trees}
\label{sub:LR}
Let $\sS[X]$ be the set of pairs $(T_1,T_2)$ of total orders,
which correspond to $213$-avoiding permutations,
i.e. in  $X$ there are no $x <_1 y <_1 z$ such that $y <_2 x <_2 z$.
If we let $\pi_1(T_1, T_2) = T_1$ and $\pi_2(T_1, T_2) = T_2$
be the projection maps as in
Subsection \ref{sub:eks-MR1}, this will not work. The problem is
that even if the pairs in the lower and right of the diagram
\eqref{eq:species-pb} are $213$-avoiding, the extension to $\sS[X]$ may
not be.

Instead we do as in Subsection \ref{sub:eks-MR2} above. Let
\[ \leq_1 = \pi_1(T_1,T_2) = T_1 \vee T_2^\op,  \quad \leq_2 = \pi_2(T_1,T_2) =
  T_1 \wedge T_2. \]
One may check for the diagram \eqref{eq:species-uniq1} that the
extension $(T_1,T_2)$ defined as in Subsection \ref{sub:eks-MR2}
is $213$-avoiding when all elements in the right diagram of
\eqref{eq:species-uniq1}
are $213$-avoiding. This is due to the cut $(A \sqcup B, C \sqcup D)$
being a global descent of the permutation corresponding to $(T_1,T_2)$.

Letting $\sB^1$ be the bimonoid species
$(\sS, \Delta^1, \mu_2, \epsilon^1, \iota_2)$
then $\clK(\sB^1)$ is by \cite{Fo-PP} seen to be Foissy's Hopf
algebra of planar trees, which is again isomorphic to Loday-Ronco's (LR)
Hopf algebra of planar binary trees. Here the correspondence is
that if $(T_1, T_2)$ corresponds to $\sigma$, the pair
maps to the basis element $M_\sigma$ in \cite{AgSo-LR}.
\begin{remark}
In \cite{AgSo-LR} there is also another basis for LR-algebra, the $F$-basis given
by \[ F_\sigma = \sum_{\tau \geq \sigma} M_\tau. \] While
we saw in the above Subsection \ref{sub:eks-MR1}
that the $F$-basis had a combinatorial (species) interpretation for the
Malvenuto-Reutenauer algebra, it seems not to be the case 
that the $F$-basis for the LR-algebra can be interpreted in
terms of cuts in species over preorders.
\end{remark}

\subsection{Hopf algebras of permutations with avoidance}
Let $\sS$ again be the restriction species of pairs $(T_1,T_2)$ of total orders,
and again we make this into two species over preorders by:
\[ \pi_1(s) = T_1 \vee T_2^\op, \quad \pi_2(s) = T_1 \wedge T_2. \]
We saw in Subsection \ref{sub:eks-MR2} that this gives two
intertwined restriction comonoids $\Delta^1$ and $\Delta^2$. 
Let $\sA$ be a subspecies (not necessarily restriction subspecies)
consisting of pairs $(T_1,T_2)$ such that the associated permutation
has {\it no global descent}. This is equivalent to
$T_1 \vee T_2^\op$ being the coarse preorder $C$.
Let $A$ be the set of permutations coming from the elements of $\sA$.
They are permutations without global descents and may be of varying length.

\begin{theorem} \label{thm:perm-A} Let $A$ be any set of
  permutations (of possibly different lengths) with {\it no global descents}.
  The $A$-avoiding permutations form a Hopf algebra which is
  a quotient Hopf algebra of the Malvenuto-Reutenauer Hopf algebra.

  Formulated for species, let $\sS$ be the species of pairs $(T_1,T_2)$
  of total orders and $\sB^1$ the associated bimonoid species from
  Subsection \ref{sub:eks-MR2}. Let $\sA$ the sub-species 
  corresponding to the permutations in $A$. 
  The coproduct $\Delta^1$ for $\sS[X]$ is $\sA$-irreducible.
  Hence the $\sA$-avoiding subspecies $\sSA$ of $\sS$ gives
  a quotient bimonoid species $\sBA^1$ of $\sB^1$.
\end{theorem}

  Applying the Fock functor we get the above statement for algebras:
  $\clK(\sBA^1)$ is a  Hopf quotient algebra of the
   Malvenuto-Reutenauer algebra
  $\clK(\sB^1)$ consisting of $A$-avoiding permutations.

\begin{proof}
  Let $\sigma = \sigma_1 \sigma_2 \ldots \sigma_n$
  be the permutation corresponding to $(T_1,T_2)$.
  The total preorder $T_1 \vee T_2^\op$ has cuts corresponding precisely
  to the global descents of $\sigma$. Let  $(U,V)$ a cut for
  $T_1 \vee T_2^\op$ corresponding to a global descent of $\sigma$ at position
  $k$. Then the restrictions $\sigma_U$ is the standardization of
  $\sigma_1 \ldots \sigma_k$ and
  $\sigma_V$ the standardization of $\sigma_{k+1} \ldots \sigma_n$.
  If $\tau$ is a permutation with {\it no global descent}
  and $\sigma$ has the pattern $\tau$, then we cannot have part
  of the $\tau$-pattern in $\sigma_U$ and another part in $\sigma_V$.
  So the full $\tau$-pattern is either in $\sigma_U$ or in $\sigma_V$.
  Hence if $\sigma$ has a $\tau$-part, then either $\sigma_U$
  has a $\tau$-part, or $\sigma_V$ has a $\tau$-part.
  That is, $\Delta^1$ is $A$-irreducible.
\end{proof}

\begin{example}
  We saw in Subsection \ref{sub:LR} that the Loday-Ronco
  Hopf algebra is the quotient algebra of the MR-algebra
  consisting of $213$-avoiding permutations.
  \end{example}

\begin{example}
  If $A = \{ 213, 132\}$, the permutations are of the type
  \[ 6789\, 45\, 123 \]
  where between each global descent, we have an increasing sequence
  with increments of size $1$.
  This can be identified as the composition $(4,2,3)$.
  In this case we get the Hopf algebra of quasi-symmetric functions.
\end{example}

\begin{example}
  Let $A = \{ 12 \}$. Then there is only a single permutation of each
  length, permutations like $54321$.
  We get the divided powers
  Hopf algebra of the polynomial ring in one variable $k[x]$.
  If we apply the Fock functor to the dual bimonoid species
  $\sB^2 = (\sS, \mu^1, \Delta^2, \iota^1, \epsilon^2)$ we get the
  ordinary  polynomial  Hopf algebra on $k[x]$. Alternatively we get the latter
  polynomial Hopf algebra by
  the contragredient Fock functor $\clK^\vee(\sB^1)$. 
\end{example}

\begin{example}
  Let $A = \{3142, 2413 \}$. This gives the Hopf algebra
  ${\mathcal W}{\mathcal NP}$ of \cite[Def.4]{Fo-PP} of plane posets avoiding
  subposets of types
  
\begin{center}
\begin{tikzpicture}[inner sep=0pt, scale=0.4]
\draw [help lines, white] (-2,-2) grid (4,2);

\node (aa) at (-1,-1) {$\bullet$};
\node (cc) at (-1,1) {$\bullet$};
\node (bb) at (1,-1) {$\bullet$};
\node (dd) at (1,1) {$\bullet$};

\draw (aa)--(cc); 
\draw (cc)--(bb);
\draw (bb)--(dd);

\node (aa) at (3,-1) {$\bullet$};
\node (cc) at (3,1) {$\bullet$};
\node (bb) at (5,-1) {$\bullet$};
\node (dd) at (5,1) {$\bullet$};

\draw (aa)--(cc); 
\draw (aa)--(dd);
\draw (bb)--(dd);
\end{tikzpicture}
\end{center}
\end{example}

\begin{problem}
  {\it  The Malvenuto-Reutenauer and Loday-Ronco Hopf algebras
  are self-dual, while the Hopf algebra of quasi-symmetric
  functions is not. Which sets $A$ of permutations without global
  descents give a quotient Hopf algebra which is self-dual?}
\end{problem}

\begin{remark} A packed word is a sequence $a_1a_2 \ldots a_n$ of
  positive integers such that the set $\{ a_1, a_2, \ldots,  a_n \}$ equals
  $\{1,2, \ldots, k\}$ for some $k$. It can be identified with
  a surjection $f : [n] \pil [k]$. It can also be identified (up to
  isomorphism) with a pair $(T,S)$ where $T$ is a total order on a set
  $X$ and $S$ is a total preorder on $X$.

  The Hopf algebra of packed words {\bf WQSymm} 
  \cite{ThNoW}, is recently considered in \cite{BV}. The
  $212$-avoiding packed words give the generalized Stirling permutations of
  \cite{GS78}, and give a Hopf algebra studied
  in \cite{BV}. As above, it is a quotient Hopf algebra
  of {\bf WQSymm}, the Hopf algebra of packed words.
  Furthermore $212$ and $213$-avoiding
  packed words may be identified with planar trees giving a Hopf algebra
  {\bf PTrees} in \cite{BV}, a quotient of the Hopf algebra of
  generalized Stirling permutations.
\end{remark}

\section{Hopf algebras of parking filtrations}
\label{sec:parking}
J-C.Novelli
and J-Y. Thibon \cite{ThNo1, ThNo2} introduced  the Hopf algebra
of parking functions {\bf PQSymm}. It has the Malvenuto-Reutenauer
Hopf algebra {\bf FQSymm} as a quotient Hopf algebra. We give a new
master Hopf algebra, sitting above both these. Its elements
are pairs $(\{X_i\}, \{Y_j\})$ of {\it parking filtrations} of $X$.
Parking functions is the Hopf sub-algebra where the filtration $\{Y_j\}$ is
equivalent to a total order on $X$.

\subsection{Parking functions and filtrations}
Let $[n] = \{1,2,\ldots, n\}$.
A {\it non-decreasing parking function} is an order preserving map
$a : [n] \pil [n]$ such that $a(i) \leq i$. It may be identified
as non-decreasing sequences $a_1, a_2, \ldots, a_n$ such that
$a_i \leq i$. A {\it parking function} is a sequence
$a_1, \dots, a_n$ such that if ordered in non-decreasing order, it
becomes a parking function as above. So if an integer $q$ occurs, there
are at least $q-1$ of the $a_i's$ which are $< q$.

\begin{definition}
  For a finite set $X$ with cardinality $n$,
a {\it parking function} is a function
$p : X \pil \NN$ such that for $i \leq n$,
then $p^{-1}([i])$ has cardinality $\geq i$. Equivalently
(let $X_i = p^{-1}([i])$), if a filtration of $X$:
    \[ \emptyset = X_0 \sus X_1 \sus \cdots \sus X_p \sus \cdots  \]
    has $|X_p| \geq p$ for $0 \leq p \leq n$, we call this a {\it parking
      filtration}.
\end{definition}


In \cite{ThNo1, ThNo2}, J-Y. Thibon and J-C. Novelli introduced
Hopf algebras of parking functions. To define such a Hopf algebra,
they needed to be able to ``standardize'' any sequence to a parking
function. For our purpose we need to standardize
any {\it exhaustive} filtration of $X$ (meaning that $X_p = X$ for
large $p$)
\[\emptyset = X_0 \sus X_1 \sus \cdots \sus X_p \sus \cdots \]
to a parking filtration. We assume in the following that our filtrations
are always exhaustive.

\begin{definition}
  Let $\{X_t\}$ be a filtration of $X$ with $n = |X|$.
  Define an order-preserving function $p : [n] \pil \NN$ by induction as
  follows. First, as convention, set $p(0) = 0$. If $p(t-1)$ is defined for $1 \leq t \leq n$,
  let $p(t)$ be the minimal $p > p(t-1)$ such that $|X_p| \geq t$.

  We thus get a strictly increasing sequence
  $p(0) < p(1) < \cdots < p(n)$, the {\it dilation} sequence
  for $\{X_t\}$. It measures how large the index $p$ must be so
  that $|X_p| \geq t$, while also keeping the sequence $\{p(t)\}$
  strictly increasing.

  Note that $\{X_t\}$ is a parking filtration iff
  $p(t) = t$ for each $t$. Let $\{\oX_t\}$ be the filtration given by
  $\oX_t = X_{p(t)}$. Then $\{\oX_t\}$ is a parking filtration, the
  {\it parkization} of $\{X_t\}$ and we also write $\{X_t\}^- =
  \{\oX_t\}$. Two filtrations are {\it equivalent} if they have the same
  parkizations.
\end{definition}

\begin{definition} \label{def:park-order}
  For a filtration $\{ X_t\}$, the set of $b$ in $\{0\} \cup [n]$ such
  that the cardinality $|X_{p(b)}| = b$, are the {\it break points} of
  $\{X_t\}$. Let $B$ be the set of break points. Then $B \supseteq \{0,n\}$.
  There is a total preorder on $X$ given by $y \geq x$ if for every $b \in B$,
  $y \in X_b$ implies $x \in X_b$. 
  If we write: 
  \[ B : 0 = b_0 < b_1 < \cdots < b_\ell = n, \]
  this is the total preorder on $X$ whose bubbles (see Definition
  \ref{def:pre-c}) are the $X_{p(b_r)} \backslash X_{p(b_{r-1})}$ with these
  successively larger for the preorder as $r$ increases. 
\end{definition}

\begin{remark} \label{rem:parking-Hp}
  Total orders on $X$ correspond to those parking filtrations where
 all points are break points,  $B = \{0 \} \cup [n]$.
  
  A parking function as defined in the beginning of this subsection,
  corresponds to a pair $(\{X_i\}, \{Y_j\})$ of parking filtrations
  of a set $X$, such that $\{Y_j\}$ is a parking filtration where
  all points are break points. So $\{Y_j\}$  gives a total order on
  $Y_n = X = X_n$ and this may be identified with $[n]$. Then
  $X_t \backslash X_{t-1}$ are
  those elements $i$ in $X = [n]$ such that $a_i = t$.
\end{remark}

Our goal now is to show:
\begin{proposition} \label{pro:parking-UX}
  Let $\{X_t \}$ and $\{Y_t \}$ be equivalent filtrations of $X$
  and $U \sus X$. Then $\{U \cap X_t\}$ and
    $\{ U \cap Y_t \}$ are equivalent filtrations
    of $U$.
  \end{proposition}

  We do this by successively removing elements:
  
    \begin{proposition} \label{pro:parking-xX} {} \hskip 2mm
    Let $\{X_t \}$ and $\{Y_t \}$ be equivalent filtrations of $X$.
    Let $x \in X$. Then $\{X_t \backslash \{x\} \}$ and
    $\{ Y_t \backslash \{x\} \}$ are equivalent filtrations
    of $X \backslash \{x \}$.
  \end{proposition}

  We first show a lemma.
  
\begin{lemma}
\label{lem:parking-p}
  Let $\{X_t\}$ be a filtration of $X$ with dilation
  function $p^o$. Let $b ^\prime < b$ be successive break points,
  and $x \in X_{p^o(b)} \backslash X_{p^o(b^\prime)}$. The dilation
  function $p^n$ of the filtration $\{X_t \backslash \{x\} \}$
  of $X\backslash \{x\}$ is:
  \[ p^n(t) = \begin{cases} p^o(t), & t < b \\
      p^o(t+1), & t \geq b
    \end{cases}. \]
  \end{lemma}

    \begin{proof}
      Clearly for $t \leq b^\prime$ we have $p^n(t) = p^o(t)$.
      Let $b^\prime < t < b$ so this $t$ is not a break point for
      $\{X_t\}$. Then
      \begin{itemize}
      \item $p^n(t)$ is the smallest $p > p^n(t-1) = p^o(t-1)$ (these are
        equal by induction) such that
      $X_p \backslash \{x\}$ has cardinality $\geq t$. 
\item  $p^o(t)$ is the smallest $p > p^o(t-1)$ such that
      $X_p$ has cardinality $\geq t$. 
    \end{itemize}
    Since $t$ not a break point, $X_{p^o(t)}$ has cardinality $\geq t+1$.
    Hence  $X_{p^o(t)} \backslash
    \{x \}$ has cardinality $\geq t$, and so $p^n(t) = p^o(t)$.
    
Consider $t = b$. Then 
\begin{itemize}
\item $p^n(b)$ is the smallest $p > p^n(b-1) = p^o(b-1)$ such that
      $X_p \backslash \{x\}$ has cardinality $\geq b$. 
\item  $p^o(b)$ is the smallest $p > p^o(b-1)$ such that
      $X_p$ has cardinality $\geq b$. 
    \end{itemize}
    Since $b$ is a break point $|X_{p^o(b)}| = b$ and so
    $|X_{p^o(b)} \backslash \{x\}| = b-1$. Thus we must have
    $p^n(b) > p^o(b)$. Now the smallest $p > p^o(b)$ such that
    $X_p \backslash \{x\}$ has cardinality $\geq b$ is the smallest
    $p > p^o(b)$ such that $|X_p| \geq b+1$. So $p^n(b) = p^o(b+1)$.

    Not let $t > b$. Then $p^n(t)$ is the smallest $p > p^n(t-1) = p^o(t)$ such
    that $X_p \backslash \{x\}$ has cardinality $\geq t$.
    This is again the smallest $p > p^o(t)$ such that $|X_p| \geq t+1$,
    and so $p^n(t) = p^o(t+1)$.
  \end{proof}

  \begin{proof}[Proof of Proposition \ref{pro:parking-xX}]
  Let $\{X_t \}$ have dilation function $p^o$ and $\{Y_t\}$ dilation function
  $q^o$. For $0 \leq t \leq n$ we have $X_{p^o(t)} = Y_{q^o(t)}$.
  The break points of these are the same.

  Let $\{X^\prime_t\}$ and $\{Y^\prime_t\}$ be the parkizations of 
$\{X_t \backslash \{x\} \}$ and
  $\{ Y_t \backslash \{x\} \}$. Let $x$ be in in
  $X_{p^o(b)} \backslash X_{p^o(b^\prime)} = Y_{q^o(b)} \backslash Y_{q^o(b^\prime)}$. 
For $t < b$ we have:
\[ X^\prime_t =  (X \backslash \{x\})_{p^n(t)} =
  X_{p^o(t)} \backslash \{x \} = Y_{q^o(t)} \backslash \{x\} 
  = (Y \backslash \{x\})_{q^n(t)} = Y^\prime_t. \]
The case when $t \geq b$ is similar.  
\end{proof}

We shall also need the following.

  \begin{lemma} \label{lem:parking-finer}
    The total preorder on $X\backslash \{x\}$ induced by $\{X_t \backslash
    \{x\} \}$ is finer than the restriction to $X \backslash \{x\}$
    of the total preorder on $X$ induced by $\{X_t \}$.
  \end{lemma}

  \begin{proof}
    We must show that each bubble of $X \backslash \{x\}$ is contained
    in a bubble of $X$.
    Let $b^\prime < b$ be successive break points for $\{X_t\}$ such
    that $x$ is in $X_{p^o(b)} \backslash X_{p^o(b^\prime)}$.
    First note that:
    \begin{itemize}
      \item The break points $\leq b^\prime$
    are the same for $\{X_t\}$ and $\{X_t \backslash \{x \} \}$. 
  \item If $c \geq b$,
    then $c$ is a break point for $\{X_t \backslash \{x \} \}$
    iff $c+1$ is a break point for $\{X_t\}$ by Lemma \ref{lem:parking-p}.
  \end{itemize}

    Let $c^\prime < c$ be
    successive break points of $X \backslash \{x \}$.
    If $c \leq b$ or $b \leq c^\prime $ it readily follows that
    the associated bubble for $\{X_t \backslash \{x \} \}$ is
    contained in the associated bubble for  $\{X_t\}$. 

    So assume $c^\prime < b \leq c$. Consider $X_{p^o(b)-1} \sus X_{p^o(b)}$.
    If this is an equality they have cardinality $b$ and
    by construction of $p^o(b)$ we must have $p^o(b-1) = p^o(b)-1$. 
    By Lemma \ref{lem:parking-p} $p^n(b-1) = p^o(b-1) = p^o(b)-1$
    and $b-1$ is a break point for
    $\{X_t \backslash \{x \} \}$. Whence $c^\prime = b-1$, and
    it follows that the associated bubble for
    $\{ X_t \backslash \{x\}\}$ is contained in a bubble for $\{X_t\}$.

    If $X_{p^o(b)-1}$ is strictly contained in $X_{p^o(b)}$, then
    since $p^o(b-1) < p^0(b)$, its cardinality must be $b-1$.
    Then $b-1$ must be break point (but note that possibly $p^o(b-1)
    < p^o(b)-1$). So $b^\prime = b-1$ and $X_{p^o(b)-1} = X_{p^o(b)} \backslash
    \{x\}$.  Letting
    $b^{\prime \prime}$ be the successor of $b$ among the break points
    for $\{X_t\}$, we have  $c^\prime = b-1$ and $c = b^{\prime \prime} -1$ and
    we have containment of the associated bubbles.
 \end{proof}


\begin{definition} Let $\{X_t\}$ be a parking filtration.
When $b$ is a break point
\[ \emptyset = X_0 \sus X_1 \sus \cdots \sus X_b,  \quad
  \emptyset = X_b \backslash X_b \sus
  X_{b+1} \backslash X_b \sus \cdots \sus X_n \backslash X_b  \]
are both parking filtrations, which we denote respectively
$\xlp{b}$ and $\xgp{b}$.
\end{definition}

\subsection{Bimonoid species of parking filtrations}

Let $\sPa[X]$ be the set of pairs $s = (\{X_i\}, \{Y_j\})$ of
parking filtrations on $X$. 
When $U \sus X$, we then get a pair $(\{U \cap X_i \}^{-}, \{U \cap Y_j\}^{-})$
where the index bar denotes parkization.
Proposition \ref{pro:parking-UX} gives that $\sPa$
is a restriction species.

Let $\pi_1(s)$ be the total preorder
associated to $\{X_i\}$ and $\pi_2(s)$ the total
preorder associated to $\{Y_j\}$, as in Definition \ref{def:park-order}.
Let $(U,V)$ be a cut for $\pi_1(s)$.
It corresponds to a break point $b$ for $\{X_i\}$,
and so $U = X_b$ and $V = X \backslash X_b$. 
We get
\[ \Delta^1 : \sPa[X] \pil \sPa[U] \times \sPa[V],
  \quad \Delta^1(s) = (\xlp{b}, \{U \cap Y_j\}^{-} ;
  \xgp{b}, \{V \cap Y_j\}^{-}). \]
Similarly $\Delta^2(s)$ is defined when $(U,V)$ is a cut for $\pi_2(s)$.

\medskip
\noindent{\it Species over preorders:}
 Lemma \ref{lem:parking-finer} shows that
 $\pi_1(\{U \cap X_j\}^-, \{U \cap Y_j\}^-)$
 is finer than the restriction $\pi_1(\{X_j\}, \{Y_j\})_{|U}$.
 Similarly for the second projection $\pi_2$.

 Furthermore if $b$ is a break point for $\{X_t\}$ and $U = X_b$,
 the total preorder
 from $\{X_j\}$ restricted to $U$, equals the preorder from $\xlp{b}$.
 Similarly if $V = X \backslash X_b$, the total preorder from $\{X_j\}$
 restricted to $V$ also equals the preorder from $\xgp{b}$.
 The situation for the second projection $\pi_2$ is similar.

\medskip
\noindent{\it Extension:} Given the diagram
\eqref{eq:species-uniq1}, we can construct $\{ X_j\}$
from $\frac{s_{AB}}{s_{CD}}$ since $A \cup B = X_p$ where $p$ is the cardinality
of $A \cup B$, and $p$ must be a break point
for $\{X_j\}$. Similarly we can construct $\{ Y_j\}$ from
$s_{AC} | s_{BD}$ as $A \cup C = Y_q$ where $q$ is the cardinality of
$A \cup C$ and $q$ must be a break point for $\{Y_j\}$. 

\medskip
\noindent {\it Cuts:} It is clear by construction of $\{X_j\}$ that
$A \cup B = X_b$ where $b$ is a break point for $\{X_j\}$.
The same goes for $\{Y_j\}$ and $A\cup C$.

\medskip
We get the following bimonoid species and Hopf algebra:
\begin{theorem}
  The species $\sPa$ over $\setmm$ where $\sPa[X]$ are pairs of parking
  filtrations on $X$, has two natural intertwining comonoids
  $\Delta^1, \Delta^2$ coming from the associated total preorders of
  the pairs
  
  Thus $(\sPa,\Delta^1, \mu_2)$ becomes a bimonoid in species and
  by the Fock functor we get  a Hopf algebra $\clK(\sPa)$ consisting
  of isomorphism classes of pairs of parking filtrations.
  \end{theorem}

\subsection{Sub- and quotient Hopf algebras by avoidance}
When we restrict to filtrations $\{Y_p\}$ of $X$ where each $p$ is a break
point, $|Y_p| = p$, this
corresponds to a total order on $X$. If $\sA$ is the subspecies
of pairs of parking filtrations where the second filtration
$\{Y_p\}$ has some $p$ with $|Y_p| > p$, the coproduct $\Delta^2$
is $\sA$-irreducible. Thus by Proposition \ref{pro:avoid-subquot}
we get a quotient bimonoid species
$\sPaA$ of $\sPa$. By Remark \ref{rem:parking-Hp}
the associated Fock functor $\clK(\sPaA)$
is the Hopf algebra of parking functions
{\bf PQSymm} of Thibon et.al. \cite{ThNo2}, which
now is realized as a quotient Hopf algebra of the larger Hopf
algebra $\clK(\sPa)$.

If we further let $\sA^\prime$ be the subspecies of $\sPaA$
consisting of parking filtrations
where there is a $q$ with
$|X_q| > q$, the coproduct $\Delta^1$ is $\sA^\prime$-irreducible.
As pairs of parking filtrations avoiding $\sA$ and $\sA^\prime$
corresponds to pair of total orders, this gives by Proposition
\ref{pro:avoid-subquot} the Malvenuto-Reutenauer
Hopf algebra as a sub Hopf algebra of the Hopf algebra of parking
functions.

\vskip 1.5cm
{\noindent \large \bf Part III: Hopf algebras of pairs of preorders.}

\section{Bimonoid species of pairs of preorders}
\label{sec:pairs}

We consider restriction species
$\sS : \setii \pil \set$
where $\sS[X]$ consists of pairs $(P,Q)$ of preorders on $X$.
This gives two natural structures as species over preorders,
by the two projections $\pi_1(P,Q) = P$ and $\pi_2(P,Q) = Q$.

We want to find such species for which the two associated
restriction comonoids are intertwined. By dualizing, say $\Delta^2$,
this gives new master bimonoid species and Hopf algebras.

\subsection{Systematic ways of extending}
Consider the diagram \eqref{eq:species-uniq1}.
To lift to a pair of preorders $(P,Q) = (\leq_1, \leq_2)$ on $X$, 
we must for each $x,y \in X$ determine the single way to 
compare them for the preorders $\leq_1$ and $\leq_2$, since there must
be exactly one such extension pair.

      If $x,y \in A\scup B, C\scup D, A \scup C$ or $B \scup D$, this
      information is already given in the diagram \eqref{eq:species-pb}.
      The only cases where it is not already determined are if they are in
      $A$ and $D$
      or in $B$ and $C$, say (for convenience rename $x$ and $y$
      to i. $a$ and $d$ or ii. $b$ and $c$):
\begin{itemize}
\item[i.] $a \in A$ and $d \in D$, or
\item[ii.] $b \in B$ and $c \in C$.
\end{itemize}
So we need to determine how $a,d $ and
$b,c$ should be compared in these cases. The extension should be unique.


\begin{lemma} \label{lem:pairs-ABCD}
  Given two preorders $\leq_1$ and $\leq_2$, and
  cut $(D_1,U_1)$ for $\leq_1$ and $(D_2,U_2)$ for $\leq_2$. Denote
 \[ A = D_1 \cap D_2,  \quad B = D_1 \cap U_2,
\quad C = U_1 \cap D_2, \quad D = U_1 \cap U_2. \] 

  a. Let $a \in A, d \in D$. Then they are not in the same
  $\geq_1$-bubble and not in the same $\geq_2$-bubble. If they are
  comparable for $\geq_1$, then $d >_1 a$. If they are comparable
  for $\geq_2$ then $d >_2 a$.

  b. Let $b \in B, c \in C$. Then they are not in the same
  $\geq_1$-bubble and not in the same $\geq_2$-bubble. If they are
  comparable for $\geq_1$, then $c >_1 b$. If they are comparable
  for $\geq_2$ then $b >_2 c$.
\end{lemma}

\begin{proof} This is simple to see from the definitions of $A,B,C,D$.
\end{proof}

      The systematic ways to postulate how to compare $a,d$ 
      are  then the following four ways.
      For every $A,D$ obtained from cuts we consistently require exactly
      one of the following cases:
      
      \begin{itemize}
      \item Always let $d >_1 a$ and $d >_2 a$.
      \item Always let $d >_1 a$ and $d \nrel{2} a$.
        \item Always let $d \nrel{1} a$ and $d >_2 a$.
        \item Always let $d \nrel{1} a$ and $d \nrel{2} a$. 
      \end{itemize}
      
  The systematic ways to postulate how to compare $b,c$ 
      are the  following four ways. For every $B,C$ obtained from cuts as above:
      
      \begin{itemize}
      \item Always let $c >_1 b$ and $b >_2 c$.
      \item Always let $c >_1 b$ and $b \nrel{2} c$.
        \item Always let $c \nrel{1} b$ and $b >_2 c$.
      \item Always let $c \nrel{1} b$ and $b \nrel{2} c$. 
      \end{itemize}

For $\sim$ a relation on $X$ and two subsets $Y,Z$ of $X$, write
$Y \sim Z$ if $y \sim z$ for every $y \in Y$ and $z \in Z$.
We will use for instance $Y >_1 Z$, $Y \geq_1 Z$, $Y \nrel{1} Z$,
or $Y \circ_1 Z$. We now use notation as in Lemma \ref{lem:pairs-ABCD}.
  
\begin{lemma} \label{lem:types-gg}
  Consider arbitrary cuts $(D_1,U_1)$ for $\leq_1$ and $(D_2,U_2)$ for
  $\leq_2$.
\begin{itemize}
\item[a.] If for such cuts we always have $A <_1 D$,
  then we always have $B <_1 C$.
\item[b.] If for such cuts we always have $A <_2 D$,
  then we always have $C <_2 B$.
\end{itemize}
\end{lemma}

\begin{proof}
  a. Let $b \in B$ and $c \in C$. Either they are i. not comparable
  $b \nrel{1} c$
  or ii. $c >_1b$. In the first case i, since they are not in the same
  $\geq_2$-bubble
there is $\geq_2$ cut separating them, say $c \in U^\prime_2$ and
$b \in D^\prime_2$.
But since $b$ and $c$ are not $\geq_1$-comparable, we can suitably get
$c \in U^\prime_1$ and $b \in D^\prime_1$. Then $c \in D^\prime$
and $b \in A^\prime$ and so $c >_1 b$, contradicting $b \nrel{1} c$.
Hence we must have case ii: $c >_1 b$. The argument for part b. is similar.
\end{proof}

\subsection{Preorders of type $\cc$}
We now want to characterize preorders $\geq_1$ and $\geq_2$
such that always:
\[ A <_1 D \text{ and } A <_2 D. \]
Such a pair is said to be of type $\cc$. 

For a preorder and $x \in X$ we write $U(x)$ for the set
$\{y \, | \, y \geq x \}$, the up-set generated by $x$.
Recall the notation $x \circ_i y$ from Definition \ref{def:pre-c},
related to the preorder $\leq_i$.

\begin{proposition} \label{pro:pairs-cc}
A pair of preorders has always $A <_1 D$ and $A <_2 D$ iff the following
two conditions hold:

\begin{itemize}
\item[1.] $x \nrel{1} y$ implies $x \circ_2 y$
\item[2.] $x \nrel{2} y$ implies $x \circ_1 y$
\end{itemize}
\end{proposition}

\begin{proof}
Assume the pair is of type $\cc$. 
Let $x \nrel{1} y$. 
Let $(U_1,D_1)$ be a $\geq_1$-cut with $y \in D_1$ and $x \in U_1$.
Consider the cut $(U_2(x),D_2)$. If not $y \geq_2 x$ then
$x \in U_1 \cap U_2 =: D$ and $y \in D_1 \cap D_2 =: A$. But then
$x >_1 y$, a contradiction. Hence we must have $y \geq_2 x$.
Similarly $x \geq_2 y$ and so $x \circ_2 y$.

Part 2 follows similarly.

\medskip
Suppose now the conditions 1 and 2 on $x$ and $y$ hold.
Let $a \in A$ and $d \in D$. If $a \nrel{1} d$ , they are
in a $\geq_2$-bubble, contradicting $a \in D_2$ and $d \in U_2$.
Since we cannot have $d \leq_1 a$, we must then have $d >_1 a$.
Similarly we have $d >_2 a$.
\end{proof}

\subsection{Preorders of type $\nc$}

We now want to characterize preorders $\geq_1$ and $\geq_2$ such that
always
\[ A \nrel{1} D, \quad A <_2 D. \]
We denote this as type $\nc$.
By Lemma \ref{lem:types-gg}
we must have $C <_2 B$. However we have two possibilities
\[ B \nrel{1} C, \quad B <_1 C. \]
These two possibilities are for the moment denoted as types $\nc_1$ and $\nc_2$.
However we shall see that $\nc_2$ really does not occur.

\begin{lemma} \label{lem:types-nc1}
  Given a pair of preorders. If the pair is of type $\nc$ the two
 following conditions hold:
 \begin{itemize}
   \item[1.] $x >_1 y$ implies $x \leq_2 y$,
\item[2.] $x \nrel{2} y$ implies $x \circ_1 y$.
\end{itemize}
\end{lemma}

\begin{proof}
1. Suppose $x >_1 y$.
  If $x >_2 y$ or $x \nrel{2} y$, we may cut such that $x \in D$ and $y \in A$,
  which gives $x \nrel{1} y$ by type $\nc$, against assumption.
  Hence $x \leq_2 y$.

 2. Suppose $x \nrel{2} y$. If $x \nrel{1} y$
  we may cut with $\geq_1$ and $\geq_2$ such that $x \in D$ and $y \in A$
  and so $x >_2 y$, against assumption. Hence $x$ and $y$ are comparable
  for $\geq_1$.
  If $x >_1 y$, part 1 gives $x \leq_2 y$ against the assumption $x \nrel{2} y$.
  In the same way we cannot have $y >_1 x$. Thus $x$ and $y$ must
  be in the same $\geq_1$-bubble.
\end{proof}

\begin{lemma}
  If a preorder is of type $\nc_2$, then $A$ or $D$ are always
  empty, and so it is subsumed under class $\cc$.
\end{lemma}

\begin{proof}
  Suppose $x \nrel{1} y$. We can then not have $x >_2y$ as we could
  then cut such that $B$ contains $x$ and $C$ contains $y$,
  contradicting that $B <_1 C$. Similarly we cannot have $x \nrel{2} y$.
  Hence $x \circ_2 y$. Thus $x \nrel{1} y$ implies $x \circ_2 y$.
  From the above Lemma \ref{lem:types-nc1}, if $ x \nrel{2} y$ we
  have $x \circ_1 y$. 

  Hence any two elements $x,y$ of $X$ fulfills the conditions of
  Proposition \ref{pro:pairs-cc} and so the preorder is of type $\cc$.
  As a consequence for this subclass of $\cc$ we would always
  have $A$ or $D$ empty.
  \end{proof}

  So only type $\nc_1$ possible, and we denote it as just $\nc$.

  \begin{proposition} \label{pro:pairs-nc}
    A pair of preorders is of type $\nc$ iff the following
    holds:
    \begin{itemize}
    \item[1.] $x \nrel{2} y$ implies $x \circ_1 y$,
    \item[2.] $x >_1 y$ implies $x \circ_2 y$.
    \end{itemize}
  \end{proposition}

  \begin{proof}
    If it is of type $\nc$, by Lemma \ref{lem:types-nc1}
    it fulfills part 1. If $x >_1 y$, by
    the same lemma $y \geq_2 x$. If $y >_2 x$, we could cut so that
    $y \in B$ and $x \in C$, contradicting that $B \nrel{1} C$.
    Thus $x$ and $y$ must be in the same bubble.

    \medskip
    Conversely, suppose 1 and 2. Given $a \in A$ and $d \in D$. By
    part 1 we must have $a <_2 d$. If $a <_1 d$ then by part 2 they are in the
    same $\geq_2$-bubble, a contradiction. So $a \nrel{1} d$. Further
    let $b \in B$ and $c \in C$. If $b <_1 c$ they are by part 2 in the
    same $\geq_2$-bubble, which they cannot be. Thus $b \nrel{1} c$.
  \end{proof}
  We can also have preorders $(\leq_1,\leq_2)$ of type $\cn$ but by
  switching them to $(\leq_2,\leq_1)$ we get a pair of type $\nc$.
    
  \subsection{Preorders of type $\nn$}
We now want to characterize preorders $\geq_1$ and $\geq_2$ such that
always
\[ A \nrel{1} D, \quad A \nrel{2} D. \]
We denote such a pair of preorders  as type $\nn$.
For $B$ and $C$ we have the possibilities
\[ i)\,\, B \nrel{1} C \text{ or } C >_1 B,  \quad ii)\,\, B \nrel{2} C \text{ or }
  B >_2 C\]
We shall show that only the case
\[B \nrel{1} C,\quad B \nrel{2} C\] occurs.

\begin{lemma} \label{lem:types-nn}
  Given a pair of preorders. If it is of type $\nn$ the following holds:
  \begin{itemize}
  \item[1.] $x>_1 y$ implies $x \leq_2 y$,
  \item[2.] $x >_2 y$ implies $x \leq_1 y$.
  \end{itemize}
\end{lemma}

\begin{proof}
  Suppose $x >_1 y$. If $x \nrel{2} y$ or $x >_2 y$ we can
  cut such that $x \in D$ and $y \in A$. This contradicts that we should
  have $x \nrel{1} y$. Hence we must have $x \leq_2 y$.
  Case 2 is similar.
\end{proof}

\begin{lemma}
  Given a pair of preorders of type $\nn$. If always $C >_1 B$, then
  always either $A$ or $D$ is empty.
  Similarly if always $B >_2 C$,
  then always either $A$ or $D$ is empty.
Hence these cases are subsumed under $\cc, \nc$ or $\cn$.
  
  Hence a pair of preorders of type $\nn$ may be taken to fulfill
  \[ B \nrel{1} C, \quad B \nrel{2} C. \]
\end{lemma}

\begin{proof} Let $a \in A$ and $d \in D$. Since $a \nrel{1} d$ and
  $a \nrel{2} d$ we can cut in another way such that $a \in C$ and $d \in B$.
  But this contradicts that $B <_1 C$. Hence we cannot have both $A$ and
  $D$ non-empty. Similarly we cannot always have $C <_2 B$.
\end{proof}

\begin{proposition} \label{pro:pairs-nn}
  A pair of preorders is of type $\nn$ iff the following holds:

  \begin{itemize}
  \item[i.] $x >_1 y$ implies $x \circ_2 y$,
  \item[ii.] $x >_2 y$ implies $x \circ_1 y$.
  \end{itemize}
\end{proposition}

\begin{proof}
  Let it be of type $\nn$ and suppose $x >_1 y$. Then $x \leq_2 y$.
  If $x <_2 y$ we could cut such that $x \in C$ and $y \in B$, giving
  a contradiction since $B \nrel{1} C$. Hence $x$ and $y$ must be
  in the same $2$-bubble. Part 2 is similar.

  \medskip
  Conversely if 1 and 2 hold then clearly if $a \in A$ and $d \in D$
  we cannot have $a <_1 d$ or $a <_2 d$. So we are in case $\nn$.
\end{proof}

\subsection{Species of pairs of preorders}
We get the following species:

\begin{itemize}
\item $\sCC[X]$ is the set of pairs of preorders on $X$ of type $\cc$,
\item $\sNC[X]$ is the set of pairs of preorders on $X$ of type $\nc$,
\item $\sCN[X]$ is the set of pairs of preorders on $X$ of type $\cn$,
\item $\sNN[X]$ is the set of pairs of preorders on $X$ of type $\nn$,
\end{itemize}

These are restriction species by respectively Propositions
\ref{pro:pairs-cc}, \ref{pro:pairs-nc}, and \ref{pro:pairs-nn}.
With the natural projections to the first and second factors they
become restriction species over preorders.
Furthermore in each case the corresponding two comonoids are intertwined:
In each diagram \eqref{eq:species-uniq1} we get a unique extension
$s$ by the requirements we have on respectively $\cc, \nc, \cn$ and $\nn$:
What must be determined for each extension is how to compare $A$ and $D$, and
$B$ and $C$, and the types tell us how to do this.
So we get:

\begin{theorem}
  Each of the pairs of restriction comonoid species:
  \[ (\sCC, \Delta_1, \Delta_2), \quad (\sNC, \Delta_1, \Delta_2), \quad
   (\sCN, \Delta_1, \Delta_2), \quad 
    (\sNN, \Delta_1, \Delta_2), \]
 gives two intertwined comonoids.  Dualizing various coproducts we get four
 distinct  bimonoids in species:
 \[ (\sCC, \Delta_1, \mu_2), \quad (\sNC, \Delta_1, \mu_2), \quad
  (\sCN, \Delta_1, \mu_2), \quad 
    (\sNN, \Delta_1, \mu_2). \]
By the Fock functor we get four ``master'' 
Hopf algebras with bases pairs of preorders $(\geq_1, \geq_2)$ of
respectively types
$\cc, \nc, \nc, \nn$.
\end{theorem}

\begin{example} Consider the species $\sA$ such that $\sA[X]$ is empty
  save when $X = \{x,y\}$ has two elements: Then
  $\sA[X]$ consists 
  of the $\cc$ pairs $\geq_1, \geq_2$
  where $\geq_2$ is the coarse topology, with $x \circ_2 y$.
  The coproduct $\Delta_2$ is $\sA$-irreducible for
  this species, so by Proposition \ref{pro:avoid-subquot} we get a 
  a sub-bimonoid species $\sCC^1_{/\sA}$ of $\sCC^1 = (\sCC,\Delta_1, \mu_2)$.
  Similarly, let $\sB$ be the species with $\sB[X]$ empty save
  when $X = \{x,y\}$ has cardinality two. Then $\sB[X]$
  consists of the $\cc$ pairs
  such that $x \circ_1 y$. Then $\Delta_1$ is $\sB$-irreducible,
  and so again Proposition \ref{pro:avoid-subquot} gives a quotient
  bimodule species
  $\sCC^1_{/\sA \sB}$  of $\sCC^1_{/\sA}$.
  This species $\sCC^1_{/\sA \sB}$
  identifies as the species of pairs of total orders, and so
  gives the Malvenuto-Reutenauer Hopf algebra by the Fock functor.
\end{example}

\begin{remark}
  Double preorders $(\leq_1, \leq_2)$ have been considered in various
  contexts, perhaps explicitly first as double posets in \cite{MR-DP}.
  Their interest is the notion of {\it picture} of pairs of double
  posets, giving a Hopf pairing on the Hopf algebra of double posets.
  Pictures were originally considered in the more special setting of
  tableaux by Zelevinsky \cite{Zel}.

  Double preorders are considered
  in \cite{FMP}. They get a Hopf algebra by letting the coproduct
  of $(P,Q)$ be determined by the cuts in $P$. The product of
  two pairs $(P_1,Q_1)$ and $(P_2,Q_2)$ is $(P,Q)$ where $P$ is the
  disjoint union of $P_1$ and $P_2$, and $Q$ is the union of $Q_1$ and
  $Q_2$ such that every element of $Q_2$ is made greater than every element
  of $Q_1$. This multiplication dualizes to a coproduct
  which fits into our setting. This coproduct
  comes from a restriction species over preorders, where the preorder
  associated to a pair $(P,Q)$ is $P^\bullet\, \vee \, \tot(Q)$, the join of
  $P^\bullet$, the partition order on $X$ whose bubbles are the connected
  components of $P$, and $\tot(Q)$, the minimal total preorder $T$ such that
  $Q$ refines $T$.

  Note that the Hopf algebra of \cite{FMP} has a basis consisting of
  {\it all} double preorders. In this section our double posets $(P,Q)$
  come with requirements on how $P$ and $Q$ are interrelated. 
  But they then become a species over preorders in the simplest ways,
  by the two projections.
  \end{remark}

     
\section{Pairs of types $\cc, \nc, \nn$ and
Hopf algebras by avoidance}
\label{sec:HApairs}

We give a more visual description of pairs of preorders
of types $\cc, \nc$ and $\nn$.
Let $Q$ be a preorder and $\fB$ a subset of the bubbles in $Q$.
A preorder $P$ which is a refinement of $Q$
(see Definition \ref{def:refine-refine})
is a {\it refinement of $Q$ with
respect to $\fB$} if every bubble of $Q$ which is not in $\fB$ is
a bubble of $P$. In other words, only bubbles in $\fB$ are possibly
refined.

\subsection{Basic situation} \label{subsec:basic-HApairs}
The following situation will be a common theme.

\medskip
{\noindent \bf Basic situation.}
Let $O_1$ and $O_2$ be two partition orders. Let
$\fB_1$ be a set of bubbles in $O_1$ and $\fB_2$ a set of bubbles in
$O_2$ such that:

\begin{itemize}
\item $\fB_1$ and $\fB_2$ are disjoint, i.e. no bubble is in both these
  sets, 
\item Every element of $\fB_1$ (this element is a bubble in $O_1$)
  is contained in a bubble of $O_2$,
\item Every element of $\fB_2$ (a bubble of $O_2$)
  is contained in a bubble of $O_1$,
\end{itemize}

In particular if $O_1$ and $O_2$ are such that there are no inclusions
between any pair of bubbles from them, $\fB_1$ and $\fB_2$ must be empty.

\medskip

\subsection{Type $\cc$}

Consider the basic situation above.

\begin{example}
  In Figure \ref{Fig-TT} we start from two total preorders $T_1$ and $T_2$.
  The brown boxes represent the bubbles of these two total orders, and
  how these bubbles are ordered is indicated by arrows.
  In $T_1$ the bubbles have sizes successively $3,6,4$.
  In $T_2$ the sizes of bubbles are successively $3, 3,7$.
  The partial order $\leq_1$ is a refinement of $T_1$,
  it refines the left bubbles with $3$ and $4$ elements.
  The partial order $\leq_2$ is a refinement
  of $T_2$, it refines the right bubbles which both have three elements.
  \end{example}
\begin{figure}
\begin{center}
\begin{tikzpicture}[inner sep=1pt, scale=1.3]
\draw [help lines, white] (-2,-2) grid (4,2);
\draw[thick, brown, rounded corners ] (-1.7,-0.9) rectangle (-0.1, -0.1) {};
\draw[thick, brown, rounded corners] (-1.7,0.9) rectangle (-0.1,0.1) {};
\draw[violet, rounded corners] (0.1,-0.9) rectangle (1.7, 0.9) {};

\filldraw[gray,thin] (0.4,0.7) circle (0.05);
\filldraw[gray,thin] (0.9,0.5) circle (0.05);
\filldraw[gray,thin] (1.3,0.2) circle (0.05);
\filldraw[gray,thin] (1.1,-0.3) circle (0.05);
\filldraw[gray,thin] (0.5,-0.5) circle (0.05);
\filldraw[gray,thin] (1.5,-0.6) circle (0.05);

\node (aa) at (-1.3,0.3) {};
\node (cc) at (-1.3,0.7) {};
\node (bb) at (-0.6,0.3) {};
\node (dd) at (-0.6, 0.7) {};

\fill (aa) circle (0.06);
\fill (cc) circle (0.06);
\fill (bb) circle (0.06);
\fill (dd) circle (0.06);

\draw (aa)--(cc); 
\draw (cc)--(bb);
\draw (bb)--(dd);

\node (anv) at (-1.3,-0.3) {};
\node (cnv) at (-0.95,-0.7) {};
\node (bnv) at (-0.6,-0.3) {};

\fill (anv) circle (0.06);
\fill (cnv) circle (0.06);
\fill (bnv) circle (0.06);

\draw (anv)--(cnv); 
\draw (cnv)--(bnv);

\node (n1) at (0.7,-0.9) {};
\node (n2) at (-0.7,-0.9) {};
\node (n3) at (0.7, 0.9) {};
\node (n4) at (-0.7,0.9) {};

\path[->] (n2) edge[bend right] (n1);
\path[->] (n3) edge[bend right] (n4);

\draw[violet, rounded corners ] (-1.7+6.6,-0.9)
rectangle (-0.1+6.6, -0.1) {};
\draw[violet, rounded corners] (-1.7+6.6,0.9) rectangle (-0.1+6.6,0.1) {};
\draw[thick, brown, rounded corners] (0.1+3,-0.9) rectangle (1.7+3, 0.9) {};

\filldraw[gray,thin] (0.4+3,0.7) circle (0.05);
\filldraw[gray,thin] (0.9+3,0.5) circle (0.05);
\filldraw[gray,thin] (1.3+3,0.2) circle (0.05);
\filldraw[gray,thin] (0.7+3,0.0) circle (0.05);
\filldraw[gray,thin] (1.1+3,-0.3) circle (0.05);
\filldraw[gray,thin] (0.5+3,-0.5) circle (0.05);
\filldraw[gray,thin] (1.5+3,-0.6) circle (0.05);

\node (aho) at (5.4,0.5) {};
\node (a1ho) at (5.5,0.6) {};
\node (a3ho) at (5.35,0.35) {};

\node (bho) at (5.9, 0.5) {};

%
\draw (aho) circle (0.3);
\fill[gray, thin] (a1ho) circle (0.06);
\fill[gray, thin] (a3ho) circle (0.06);
\fill (bho) circle (0.06);


\node (ahn) at (5.3,-0.7) {};
\node (chn) at (5.65,-0.3) {};
\node (bhn) at (6,-0.7) {};

\fill (ahn) circle (0.06);
\fill (chn) circle (0.06);
\fill (bhn) circle (0.06);

\draw (ahn)--(chn); 
\draw (chn)--(bhn);

\node (n1) at (6.5,-0.3) {};
\node (n2) at (6.5,0.3) {};
\node (n3) at (5.8, 0.9) {};
\node (n4) at (3.8 ,0.9) {};

\path[->] (n1) edge[bend right] (n2);
\path[->] (n3) edge[bend right] (n4);

\end{tikzpicture}


\caption{$T_1$ with refinement $\leq_1$,
  \hskip 1.5cm $T_2$ with refinement $\leq_2$}
\label{Fig-TT}
\end{center}
\end{figure}
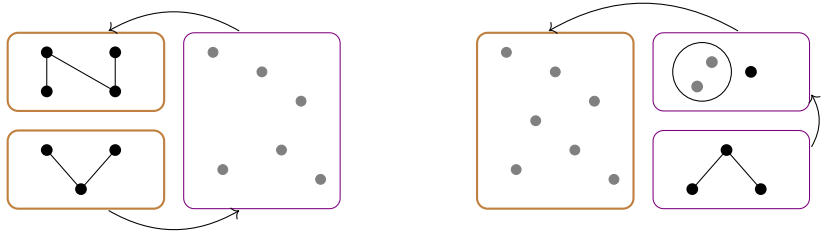

\begin{proposition} Let $O_1, B_1, O_2, B_2$ fulfill the Basic
  situation \ref{subsec:basic-HApairs}.
Let $T_1$ and $T_2$ be two total orders such that $T_1^\circ = O_1$
and $T_2^\circ = O_2$ (i.e. the bubbles of these total orders are the bubbles
of the partition orders). Let $\leq_1$ be a refinement of $T_1$ along
$\fB_1$ and $\leq_2$ a refinement of $T_2$ along $\fB_2$. 

Then the pair $(\geq_1,\geq_2)$ is of type $\cc$
and every pair of type $\cc$ arises in this way.
\end{proposition}

\begin{proof} $\Rightarrow$
  If it is constructed as above, then clearly the conditions 1 and 2 of
  Proposition \ref{pro:pairs-cc} hold.

  Conversely suppose the pair is of type $\cc$. Let $T_1$ be the
  total preorder hull of $\geq_1$, and $T_2$ the total preorder
  hull of $\geq_2$. Consider a bubble $B$ of $T_1$ which has been
  refined. By Lemma \ref{lem:prelattice-inc} the transitive closure
  of the incomparability relation $\nrel{}$
  has $B$ as the single equivalence class.
  Then by Proposition \ref{pro:pairs-cc} the bubble $B$ is contained in
  a bubble of $T_2$. Thus every bubble of $T_1$ which has been refined
  is contained in a bubble of $T_2$. Similarly every bubble of $T_2$
  which has been refined is contained in a bubble of $T_1$.
\end{proof}

\begin{example}
  Let $\sS$ be the species consisting of pairs $(T_1,T_2)$ of total
  preorders on $X$. This is the subspecies of the $\cc$-species where you avoid
  incomparable pairs of elements of $X$.
  Let $T_1$ have the sequence of bubbles $B_1, \ldots, B_n$ and
let $T_2$ have the sequence of  bubbles $C_1, \ldots, C_m$. Let
$a_{ij} = |B_i \cap C_j|$. After applying the Fock functor, we can
represent the pair $(T_1,T_2)$ by an $n \times m$ matrix $A$ with
entries $a_{ij}$, natural numbers $\geq 0$,
where $\sum a_{ij}  = |X|$. Conversely any such
matrix gives rise to a pair of preorders $(T_1,T_2)$ on $X$, up
to an automorphism of $X$.
  \end{example}

  \begin{example} \label{eks:types-TT} {\bf WQSymm.}
    Consider the subspecies where
$T_1$ is a total order and $T_2$ is a total preorder. The preorder
  $T_2$ naturally identifies as a surjection
  $X \overset{p}{\twoheadrightarrow} [k]$
with $x \leq_{T_2} y$ iff $p(x) \leq p(y)$.
The pair $(T_1, T_2)$ then corresponds to a "packed word", i.e.
a sequence $a_1, a_2, \ldots, a_n$ of natural numbers such that if
$p$ appears, and $1 \leq q \leq p$, then $q$ also appears in the sequence.
This gives the Hopf algebra of packed words, ${\bf WQSymm}$
  see \cite{ThNoW}.
\end{example}

\subsection{Type $\nc$}
If a bubble in a preorder is also a component of the preorder,
we call it a bubble component.
Consider again the basic situation.

\begin{proposition} \label{pro:types-to-nc}
 Let $O_1, B_1, O_2, B_2$ fulfill the Basic
  situation \ref{subsec:basic-HApairs}.
Let $T_2$ be a total order such that $T_2^\circ = O_2$. Let
$\leq_1$ be a refinement of $O_1$ along $\fB_1$ and $\leq_2$ a
refinement of $T_2$ along $\fB_2$. 
Then the pair is of type $\nc$ and every pair of type $\nc$ arises
in this way.
\end{proposition}


  \begin{example}
    In Figure \ref{Fig-OT} we start from a partition order $O_1$ and
    a total preorder $T_2$.
  The brown boxes represent the bubbles of these two orders, and
  how these bubbles are ordered in $T_2$ is indicated by arrows.
  In $O_1$ the bubbles have sizes $3,6,4$. In $T_2$ the sizes of bubbles
  are successively $3,3,7 $.
  The partial order $\leq_1$ is a refinement of $O_1$, it
  refines the left bubbles, and $\leq_2$ a refinement
  of $T_2$, it refines the right bubbles.
\end{example}

  \begin{figure}
\begin{center}
\begin{tikzpicture}[inner sep=1pt, scale=1.3]
\draw [help lines, white] (-2,-2) grid (4,2);
\draw[thick, brown, rounded corners ] (-1.7,-0.9) rectangle (-0.1, -0.1) {};
\draw[thick, brown, rounded corners] (-1.7,0.9) rectangle (-0.1,0.1) {};
\draw[violet, rounded corners] (0.1,-0.9) rectangle (1.7, 0.9) {};

\filldraw[gray,thin] (0.4,0.7) circle (0.05);
\filldraw[gray,thin] (0.9,0.5) circle (0.05);
\filldraw[gray,thin] (1.3,0.2) circle (0.05);
\filldraw[gray,thin] (1.1,-0.3) circle (0.05);
\filldraw[gray,thin] (0.5,-0.5) circle (0.05);
\filldraw[gray,thin] (1.5,-0.6) circle (0.05);

\node (aa) at (-1.3,0.3) {};
\node (cc) at (-1.3,0.7) {};
\node (bb) at (-0.6,0.3) {};
\node (dd) at (-0.6, 0.7) {};

\fill (aa) circle (0.06);
\fill (cc) circle (0.06);
\fill (bb) circle (0.06);
\fill (dd) circle (0.06);

\draw (aa)--(cc); 
\draw (cc)--(bb);
\draw (bb)--(dd);

\node (anv) at (-1.3,-0.3) {};
\node (cnv) at (-0.95,-0.7) {};
\node (bnv) at (-0.6,-0.3) {};

\fill (anv) circle (0.06);
\fill (cnv) circle (0.06);
\fill (bnv) circle (0.06);

\draw (anv)--(cnv); 
\draw (cnv)--(bnv);

\node (n1) at (0.7,-0.9) {};
\node (n2) at (-0.7,-0.9) {};
\node (n3) at (0.7, 0.9) {};
\node (n4) at (-0.7,0.9) {};


\draw[violet, rounded corners ] (-1.7+6.6,-0.9)
rectangle (-0.1+6.6, -0.1) {};
\draw[violet, rounded corners] (-1.7+6.6,0.9) rectangle (-0.1+6.6,0.1) {};
\draw[thick, brown, rounded corners] (0.1+3,-0.9) rectangle (1.7+3, 0.9) {};

\filldraw[gray,thin] (0.4+3,0.7) circle (0.05);
\filldraw[gray,thin] (0.9+3,0.5) circle (0.05);
\filldraw[gray,thin] (1.3+3,0.2) circle (0.05);
\filldraw[gray,thin] (0.7+3,0.0) circle (0.05);
\filldraw[gray,thin] (1.1+3,-0.3) circle (0.05);
\filldraw[gray,thin] (0.5+3,-0.5) circle (0.05);
\filldraw[gray,thin] (1.5+3,-0.6) circle (0.05);

\node (aho) at (5.4,0.5) {};
\node (bho) at (5.9, 0.5) {};

%
\draw (aho) circle (0.3);
\fill (bho) circle (0.06);

\node (a1ho) at (5.5,0.6) {};
\node (a3ho) at (5.35,0.35) {};

\fill[gray, thin] (a1ho) circle (0.05);
\fill[gray, thin] (a3ho) circle (0.05);

\node (ahn) at (5.3,-0.7) {};
\node (chn) at (5.65,-0.3) {};
\node (bhn) at (6,-0.7) {};

\fill (ahn) circle (0.06);
\fill (chn) circle (0.06);
\fill (bhn) circle (0.06);

\draw (ahn)--(chn); 
\draw (chn)--(bhn);

\node (n1) at (6.5,-0.3) {};
\node (n2) at (6.5,0.3) {};
\node (n3) at (5.8, 0.9) {};
\node (n4) at (3.8 ,0.9) {};

\path[->] (n1) edge[bend right] (n2);
\path[->] (n3) edge[bend right] (n4);

\end{tikzpicture}


\caption{$O_1$ with refinement, \hskip 2cm $T_2$ with refinement}
\label{Fig-OT}
\end{center}
\end{figure}

\begin{proof}
  It the pair is constructed above, the conditions of Proposition
  \ref{pro:pairs-nc} are fulfilled.

  Conversely, if it is of type $\nc$,
  let $T_2$ be the total preorder hull of $\geq_2$ and
  let $O_1$ be $(\geq_1)^\bullet$, the components of $\geq_1$. Then if
  $B$ is a bubble of $T_2$ which has been refined by $\geq_2$,
  by Lemma \ref{lem:prelattice-inc} and
Proposition \ref{pro:pairs-nc}.1, $B$
  is contained in a bubble of $T_1$.

  If $B$ is a bubble of $O_1$ which is  refined by $\leq_1$, by Proposition
  \ref{pro:pairs-nc}.2 $B$ is contained in a bubble of $T_2$.
\end{proof}

\begin{example}
  Let $\sS$ be the species consisting of pairs $(O_1,T_2)$ of a
  partition 
  preorder and a total preorder  on $X$.
  This is the subspecies of the $\nc$-species where you avoid
  strictly ordered elements in the first factor,
  and incomparable elements in the second factor.
  As in Example \ref{eks:types-TT}
  we may similarly construct the matrix $A$ but $A$ is only
determine up to permuting the rows. Applying the Fock functor the
image of $(O_1,T_2)$ is an equivalence class of such matrices, with matrices
equivalent if they are obtained by permuting rows.
  \end{example}

  \subsection{Type $\nn$}
  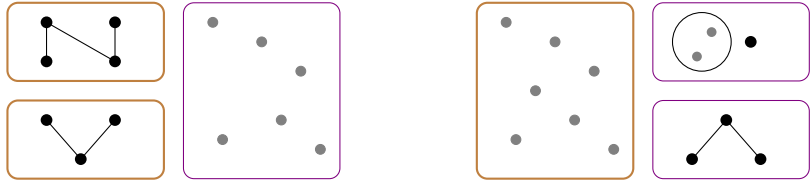
\begin{figure}
\begin{center}
\begin{tikzpicture}[inner sep=1pt, scale=1.3]
\draw [help lines, white] (-2,-2) grid (4,2);
\draw[thick, brown, rounded corners ] (-1.7,-0.9) rectangle (-0.1, -0.1) {};
\draw[thick, brown, rounded corners] (-1.7,0.9) rectangle (-0.1,0.1) {};
\draw[violet, rounded corners] (0.1,-0.9) rectangle (1.7, 0.9) {};

\filldraw[gray,thin] (0.4,0.7) circle (0.05);
\filldraw[gray,thin] (0.9,0.5) circle (0.05);
\filldraw[gray,thin] (1.3,0.2) circle (0.05);
\filldraw[gray,thin] (1.1,-0.3) circle (0.05);
\filldraw[gray,thin] (0.5,-0.5) circle (0.05);
\filldraw[gray,thin] (1.5,-0.6) circle (0.05);

\node (aa) at (-1.3,0.3) {};
\node (cc) at (-1.3,0.7) {};
\node (bb) at (-0.6,0.3) {};
\node (dd) at (-0.6, 0.7) {};

\fill (aa) circle (0.06);
\fill (cc) circle (0.06);
\fill (bb) circle (0.06);
\fill (dd) circle (0.06);

\draw (aa)--(cc); 
\draw (cc)--(bb);
\draw (bb)--(dd);

\node (anv) at (-1.3,-0.3) {};
\node (cnv) at (-0.95,-0.7) {};
\node (bnv) at (-0.6,-0.3) {};

\fill (anv) circle (0.06);
\fill (cnv) circle (0.06);
\fill (bnv) circle (0.06);

\draw (anv)--(cnv); 
\draw (cnv)--(bnv);

\node (n1) at (0.7,-0.9) {};
\node (n2) at (-0.7,-0.9) {};
\node (n3) at (0.7, 0.9) {};
\node (n4) at (-0.7,0.9) {};


\draw[violet, rounded corners ] (-1.7+6.6,-0.9)
rectangle (-0.1+6.6, -0.1) {};
\draw[violet, rounded corners] (-1.7+6.6,0.9) rectangle (-0.1+6.6,0.1) {};
\draw[thick, brown, rounded corners] (0.1+3,-0.9) rectangle (1.7+3, 0.9) {};

\filldraw[gray,thin] (0.4+3,0.7) circle (0.05);
\filldraw[gray,thin] (0.9+3,0.5) circle (0.05);
\filldraw[gray,thin] (1.3+3,0.2) circle (0.05);
\filldraw[gray,thin] (0.7+3,0.0) circle (0.05);
\filldraw[gray,thin] (1.1+3,-0.3) circle (0.05);
\filldraw[gray,thin] (0.5+3,-0.5) circle (0.05);
\filldraw[gray,thin] (1.5+3,-0.6) circle (0.05);

\node (aho) at (5.4,0.5) {};
\node (bho) at (5.9, 0.5) {};

%
\draw (aho) circle (0.3);
\fill (bho) circle (0.06);

\node (a1ho) at (5.5,0.6) {};
\node (a3ho) at (5.35,0.35) {};

\fill[gray, thin] (a1ho) circle (0.05);
\fill[gray, thin] (a3ho) circle (0.05);

\node (ahn) at (5.3,-0.7) {};
\node (chn) at (5.65,-0.3) {};
\node (bhn) at (6,-0.7) {};

\fill (ahn) circle (0.06);
\fill (chn) circle (0.06);
\fill (bhn) circle (0.06);

\draw (ahn)--(chn); 
\draw (chn)--(bhn);

\node (n1) at (6.5,-0.3) {};
\node (n2) at (6.5,0.3) {};
\node (n3) at (5.8, 0.9) {};
\node (n4) at (3.8 ,0.9) {};
\end{tikzpicture}

\caption{$O_1$ with refinement, \hskip 2cm $O_2$ with refinement}
\label{Fig-OO}
\end{center}
\end{figure}


\begin{proposition}
 Let $O_1, B_1, O_2, B_2$ fulfill the Basic
  situation \ref{subsec:basic-HApairs}.
Let $\leq_1$ be a refinement of $O_1$ along $\fB_1$ and
$\leq_2$ a refinement of $O_2$ along $\fB_2$. 
  Then the pair is of type $\nn$ and conversely any pair
  of type $\nn$ may be constructed as above.
\end{proposition}

\begin{proof}
  If it is constructed as above, it is of type $\nn$ by Proposition
  \ref{pro:pairs-nn}. Conversely if it
  is of type $\nn$, let $O_1 = ({\geq_1})^\bullet$, the partition preorder
  where the underlying sets of its bubbles are the underlying sets of the
  components of $\geq_1$, and similarly $O_2 = ({\geq_2})^\bullet$. By Proposition
  \ref{pro:pairs-nn} it is given by the Basic situation,
  \ref{subsec:basic-HApairs}.
\end{proof}


\begin{example}
  In Figure \ref{Fig-OO} the boxes represent bubbles of two partition orders
  $O_1$ and $O_2$.
 The brown boxes represent the bubbles of these two orders.
  In $O_1$ the bubbles have sizes $3,6,4$. In $O_2$ the number of bubbles
  are $3, 3, 7$.
      The partial order $\leq_1$ is a refinement of $O_1$ and $\leq_2$ a refinement
  of $O_2$.
\end{example}

\begin{example}
  Let $\sS$ be the species consisting of pairs $(O_1,O_2)$ of
  partition preorders.
  This is the subspecies of the $\nn$-species where you avoid
  strictly ordered elements in the two factors.
  As in Example \ref{eks:types-TT}
  we may similarly construct the matrix $A$
but matrices being equivalent
if they are obtained by permuting both rows and columns.
Applying the Fock functor the
image of $(O_1,O_2)$ is an equivalence class of such matrices, with matrices
equivalent if they are obtained by permuting rows and columns.
  \end{example}

\bibliographystyle{amsplain}
\bibliography{biblio}
\end{document}